\documentclass{amsart}
\usepackage{amssymb}

\usepackage{amsmath, amscd, amssymb, amsthm}

\usepackage{bbm}
\usepackage{latexsym}
\usepackage{amsfonts}
\usepackage{graphicx}
\usepackage{color}

\usepackage[
dvipdfm, colorlinks=true, linkcolor=black,breaklinks=true,
urlcolor=blue, citecolor=black
]{hyperref}%
\usepackage{geometry}
\newtheorem{theorem}{Theorem}[section]
\newtheorem{lemma}[theorem]{Lemma}

\newtheorem{definition}[theorem]{Definition}
\newtheorem{corollary}[theorem]{Corollary}
\newtheorem{example}[theorem]{Example}
\newtheorem{claim}[theorem]{Claim}
\newtheorem{conjecture}[theorem]{Conjecture}

\newtheorem{proposition}[theorem]{Proposition}
\newtheorem{question}[theorem]{Question}
\newtheorem{assumption}[theorem]{Assumption}

\newcommand{\be}{\begin{equation}}
\newcommand{\ee}{\end{equation}}
\newcommand{\bea}{\begin{eqnarray}}
\newcommand{\eea}{\end{eqnarray}}

\begin{document}

\title{Hirzebruch manifolds and positive holomorphic
sectional curvature}

\author{Bo Yang}
\thanks{Research partially supported by an AMS-Simons Travel
Grant}

\address{Bo Yang. Department of Mathematics, 310 Malott Hall,
Cornell University, Ithaca, NY 14853-4201, USA.}
\email{{boyang@math.cornell.edu} }

\author{Fangyang Zheng}
\thanks{Research partially supported by a
Simons Collaboration Grant}
\address{Fangyang Zheng. Department of Mathematics, The
Ohio State University, 231 West 18th Avenue, Columbus, OH 43210, USA
and Department of Mathematics, Zhejiang Normal University, Jinhua,
321004, Zhejiang, China} \email{{zheng.31@osu.edu}}

\begin{abstract}

This paper is the first step in a systematic project to study
examples of K\"ahler manifolds with positive holomorphic sectional
curvature ($H > 0$). Previously Hitchin proved that any compact
K\"ahler surface with $H>0$ must be rational and he constructed such
examples on Hirzebruch surfaces $M_{2, k}=\mathbb{P}(H^{k}\oplus
1_{\mathbb{CP}^1})$. We generalize Hitchin's construction and prove
that any Hirzebruch manifold $M_{n, k}=\mathbb{P}(H^{k}\oplus
1_{\mathbb{CP}^{n-1}})$ admits a K\"ahler metric of $H>0$ in each of
its K\"ahler classes. We demonstrate that the pinching behaviors of
holomorphic sectional curvatures of new examples differ from those
of Hitchin's which were studied in the recent work of
Alvarez-Chaturvedi-Heier. Some connections to recent works on the
K\"ahler-Ricci flow on Hirzebruch manifolds are also discussed.

It seems interesting to study the space of all K\"ahler metrics of
$H>0$ on a given K\"ahler manifold. We give higher dimensional
examples such that some K\"ahler classes admit K\"ahler metrics with
$H>0$ and some do not.

\end{abstract}

\maketitle

\tableofcontents

\section{Introduction}

Let $(M, J, g)$ be a K\"ahler manifold, then one can define the
holomorphic sectional curvature of any $J$-invariant real $2$-plane
$\pi=\operatorname{Span}\{X, JX\}$ by
\[
H(\pi)=\frac{R(X, JX, JX, X)}{||X||^4}.
\]
It is the Riemannian sectional curvature restricted on any
$J$-invariant real $2$-plane (p165 \cite{KN}). Compact K\"ahler
manifolds with positive holomorphic sectional ($H>0$) form an
interesting class of K\"ahler manifolds. For example, these
manifolds are simply-connected, by the work of Tsukamoto
\cite{Tsukamoto}. An averaging argument due to Berger
\cite{Berger1966} showed that $H>0$ implies positive scalar
curvature, which further leads to the vanishing of its
pluri-canonical ring by a Bochner-Kodaira type identity, see
\cite{KW}. In 1975 Hitchin \cite{Hitchin} proved that any Hirzebruch
surface $M_{2, k}=\mathbb{P}(H^{k}\oplus 1_{\mathbb{CP}^1})$ admits
a K\"ahler metrics with $H>0$. Moreover, he proved that any rational
surface admits a K\"ahler metric with positive scalar curvature.

Recall that one can define another positive curvature condition on
K\"ahler manifolds, the so-called (holomorphic) bisectional
curvature, which is stronger than $H>0$. Any compact K\"ahler
manifold with positive holomorphic bisectional curvature is
biholormophic to $\mathbb{CP}^n$ by the work of Mori \cite{Mori} and
Siu-Yau \cite{SiuYau}. Motivated by these results and Hitchin's
example, Yau \cite{Yau1} asked if the positivity of holomorphic
sectional curvature can be used to characterize the rationality of
algebraic manifolds. More precisely, the following question was
asked.

\begin{question}[Yau, Problem 67 and 68 in
\cite{Yau2}]\label{rationality}
Consider a compact K\"ahler manifold with positive holomorphic
sectional curvature, is it unirational? Is it projective? If a
projective manifold is obtained by blowing up a compact manifold
with positive holomorphic sectional curvature along a subvariety,
does it still carry a metric with positive holomorphic sectional
curvature? In general, can we find a geometric criterion to
distinguish the concept of unirationality and rationality?
\end{question}

There are some recent progress on Question \ref{rationality}. For
example, an important citerion on non-uniruledness of projective
manifolds in terms of pseudoeffective canonical line bundles was
established by Boucksom-Demailly-P{\u{a}}un-Peternell \cite{BDPP}.
Heier-Wong \cite{HeierWong2012} gave a nice application of this
result to show that any projective manifold with a K\"ahler metric
with positive total scalar curvature is uniruled, i.e., having a
rational curve passing through every point. Very recently, they
\cite{HeierWong2015} proved that any projective manifold with a
K\"ahler metric of $H>0$ is rationally connected, i.e., any two
points can be connected by a rational curve. The latter paper also
contains analogous results for $H \geq 0$ and nonnegative Ricci
curvatures. There are also some recent studies on Hermitian
manifolds with $H \geq 0$ (\cite{YangX}).

In this paper, we focus on examples of K\"ahler metrics with $H>0$.
More specifically, we want to carry out a detailed study of such
metrics on any Hirzebruch manifold $M_{n, k}=\mathbb{P}(H^{k}\oplus
1_{\mathbb{CP}^{n-1}})$.

One of our motivation is the work of Chen-Tian (\cite{ChenTian1} and
\cite{ChenTian2}), where they studied K\"ahler-Ricci flow with
positive bisectional curvature and proved that the space of all
K\"ahler metrics with positive bisectional curvature on
$\mathbb{CP}^n$ is path-connected. Similarly, we would like to
understand the space of all K\"ahler metrics with $H>0$ on any
Hirzebruch manifold $M_{n, k}=\mathbb{P}(H^{k}\oplus
1_{\mathbb{CP}^{n-1}})$.

\begin{question}\label{space of H}
Given any Hirzebruch manifold $M_{n, k}=\mathbb{P}(H^{k}\oplus
1_{\mathbb{CP}^{n-1}})$, what can we say about the space of all
K\"ahler metrics with $H>0$? Is it path-connected?
\end{question}

As a first step to answer Question \ref{space of H}, we prove the
following result.

\begin{theorem}\label{main in intro}
Given any Hirzebruch manifold $M_{n,k}$, there exists a K\"ahler
metric of $H>0$ in each of its K\"ahler classes.
\end{theorem}

Let us explain the background of Theorem \ref{main in intro} and its
relation to Hitchin's examples in \cite{Hitchin}. Recall any
Hirzebruch surface $M_{2, k}=\mathbb{P}(H^{k}\oplus
1_{\mathbb{CP}^{1}})$ can be obtained from $\mathbb{CP}^2$ by
blowing up and down, i.e. is a rational surface. It is the
projective bundle associated to the rank-$2$ vector bundle
$H^{k}\oplus 1_{\mathbb{CP}^{1}}$, $H$ being the hyperplane bundle
over $\mathbb{CP}^{1}$ and $1_{\mathbb{CP}^{1}}$ the trivial bundle.
We call a surface minimal if it has no rational curve of
self-intersection $-1$, then the only minimal rational surface are
$\mathbb{CP}^2$, $\mathbb{CP}^1 \times \mathbb{CP}^1$, and $M_{2,
k}$ with $k>1$. In other words, any rational surface is the blow up
of $\mathbb{CP}^2$, $\mathbb{CP}^1 \times \mathbb{CP}^1$, and $M_{2,
k}$ with $k>1$. We refer the reader to Chapter 4 of Griffiths-Harris
\cite{GH} for a detailed description of rational surfaces.

Now we focus on the K\"ahler metrics on Hirzebruch manifolds $M_{n,
k}$ where $\pi: M_{n,k} \rightarrow \mathbb{CP}^{n-1}$ for $n \geq
2$. Note that $M_{n,k}$ can also be described as
$\mathbb{P}(H^{-k}\oplus 1_{\mathbb{CP}^{n-1}})$, the projective
bundle associated to $H^{-k}\oplus 1_{\mathbb{CP}^{1}}$. Let $E_0$
denote the divisor in $M_{n,k}$ corresponding to the section $(0,1)$
of $H^{-k}\oplus 1_{\mathbb{CP}^{n-1}}$, $E_{\infty}$ the divisor in
$M_{n,k}$ corresponding to the section $(0,1)$ of $H^{k}\oplus
1_{\mathbb{CP}^{n-1}}$, and $F$ the divisor corresponding to the
pull-back line bundle $\pi^{\ast} H$ over $M_{n.k}$ \footnote{Note
that we follow the notations of $E_0$ and $E_{\infty}$ as in Calabi
\cite{Calabi2}, which is different with those in \cite{GH}.}. Then
the Picard group of $M_{n,k}$ is generated by the divisors $E_0$ and
$F$, while $E_{\infty}=E_0+kF$. The integral cohomology ring of
$M_{n,k}$ is
\[
{\mathbb Z}[F,E_0] / \langle F^n, E_0^2+kE_0F \rangle. \] The
anti-canonical class of $M_{n,k}$ can be expressed as
\[
K^{-1}_{M_{n,k}}=2 E_{\infty}-(k-n)F=\frac{n+k}{k}
E_{\infty}-\frac{n-k}{k} E_{0},
\] and every class $\alpha$ in the K\"ahler cone of $M_{n,k}$
can be expressed as
\begin{equation}
\alpha=\frac{b}{k} [E_{\infty}]-\frac{a}{k} [E_0].  \label{kahler
cone}
\end{equation} for any $b>a>0$.

In \cite{Hitchin} Hitchin proved that each $M_{2,k}$ admits
K\"ahler metrics with $H>0$. His example was motivated by a natural choice of
K\"ahler metric on any projective vector bundles over K\"ahler
manifold. Namely let $\pi: (E, h) \rightarrow (M,g)$ be any
Hermitian vector bundle over a compact K\"ahler manifold. The Chern
curvature form $\Theta (\mathcal{O}_{{\mathbb P}(E)} (1))$ of
$\mathcal{O}_{{\mathbb P}(E) (1)}$ over ${\mathbb P}(E)$ has the fiber direction
components given by the Fubini-Study form, hence is positive.
Therefore
\begin{equation}
\tilde{\omega}=\pi^{\ast} \omega_g+ s \sqrt{-1} \Theta
(\mathcal{O}_{{\mathbb P}(E)} (1))  \label{Hitchin intro}
\end{equation}
is a well-defined K\"ahler metric on $P(E)$ when $s>0$ is
sufficiently small. Fix a Hirzebruch surface $M_{2, k}$, one picks
$(E, h)=(H^{k} \oplus 1_{\mathbb{CP}^{1}}, h)$ and $(M, g)$ as
$(\mathbb{CP}^1, g_{FS})$ where $g_{FS}$ is the standard
Fubini-Study metric and $h$ the induced metric. Under this
situation, $\tilde{\omega}$ has a natural $U(2)$ isometric action.
Hitchin showed that $\tilde{\omega}$ has $H>0$ if $0<s(1+ks)^2
<\frac{1}{k(2k-1)}$ by an explicit calculation. It was further
observed by Alvarez-Chaturvedi-Heier \cite{ACH} that Hitchin's
examples satisfy $H>0$ if and only if $s<\frac{1}{k^2}$. We note
that Hitchin's examples can only exist on a proper open set of
K\"ahler cone of $M_{2,k}$, for example, on $M_{2,1}$, by a scaling
his examples lie in K\"ahler class $b[E_{\infty}]-aE_0$ where
$a<b<2a$. Now Theorem \ref{main in intro} implies the existence of
K\"ahler metric of $H>0$ in each of K\"ahler class on any $M_{n, k}$
where $n \geq 2$.

To prove Theorem \ref{main in intro}, we follow Calabi's ansatz
(\cite{Calabi1} and \cite{Calabi2}). The crucial observation
(pointed out in \cite{Calabi2}) that the group of holomorphic
transformations of $M_{n,k}$ contains $U(n)/Z_k$ as its maximal
compact group. Therefore it is natural to study K\"ahler metrics
with $U(n)$-symmetry. Calabi's method has been very fruitful to
produce examples of special K\"ahler metrics in various settings,
including K\"ahler-Einstein metrics, K\"ahler-Ricci solitons,
K\"ahler metrics with constant scalar curvature, extremal K\"ahler
metrics, etc., see for example \cite{KS1986}, \cite{Lebrun},
\cite{Simanca}, \cite{Koiso}, \cite{Apostolov}, \cite{Cao},
\cite{F-I-K}, \cite{HwangSinger}. To be more precise, we follow the
slightly general approach due to Koiso-Sakane \cite{KS1986}. It
turns out that for $M_{n,k}$ this method is equivalent to Calabi's
ansatz.

Let us view $M_{n,k}$ as a compactification of the
$\mathbb{C}^{\ast}$ bundle over $\mathbb{CP}^{n-1}$ obtained by
$H^{-k}\setminus E_0$ where $E_0$ is its zero section. In general,
given any $\mathbb{C}^{\ast}$ bundle $\pi: L^{\ast} \rightarrow M$
obtained by a Hermitian line bundle $(L, h)$, we may consider the
following metric on the total space of $L^{\ast}$:
\begin{equation}
\tilde{g}=\pi^{\ast}g_{t}+dt^{2}+(dt\circ \tilde{J})^{2}, \label{def
of g intro}
\end{equation} where $g_t$ is a continuous family of K\"ahler metrics
on the base $(M, J)$, $t$ is a function which only depends on the
norm of Hermitian metric $h$, and $\tilde{J}$ the complex structure
on the total space of $L$. Koiso-Sakane \cite{KS1986} gave a
sufficient condition to enure the resulting metric $\tilde{g}$ is
K\"ahler and studied the compactification of such metrics. They were
able to give new examples of non-homogeneous K\"ahler-Einstein
metrics.

Now let us focus on $M_{n,k}$, where we will pick $L=H^{-k}$ in the
approach of Koiso-Sakane. After a suitable reparametrization and
some careful analysis, one can show that $\tilde{g}$ in (\ref{def of
g intro}) can be compactified to produce K\"ahler metrics on
$M_{n,k}$ if a generating function $\phi(U)$ of a single variable
$U$ satisfying suitable boundary conditions. For simplicity, let use
still call such a metric $(M_{n,k}, \tilde{g})$. Making full use of
the $U(n)$-isometric action, it can be shown that the curvature
tensors of $(M_{n,k}, \tilde{g})$ are completely determined by three
components, namely

(1) $A$ which is the holomorphic sectional curvature along the fiber
direction $F$,

(2) $B$ which is the bisectional curvature along the
fiber and any direction in the base $E_{0}$,

(3) $C$ which is the holomorphic sectional curvature along $E_0$.

Thus we can reduce the problem of constructing K\"ahler metric with
$H>0$ in the form of $\tilde{g}$ by looking for a suitable
generating function $\phi(U)$ satisfying some differential
inequalities related to $A, B, C$ defined above. In particular, we
show that Hitchin's example is canonical among all K\"ahler metrics with
$H>0$ on $M_{n,k}$ in the following sense:

\begin{proposition} Hitchin's examples can be uniquely characterized
as $U(n)$-invariant K\"ahler metrics on $M_{n,k}$ which have the
constant curvature component $A$.
\end{proposition}

In the level of the generating function $\phi(U)$ in the setting of
Koiso-Sakane, each of Hitchin's example corresponds to some
quadratic even function defined on $[-c, c]$ with
$0<c<\frac{n}{k(2k+1)}$. Indeed, the boundary conditions of the
generating function $\phi(U)$ where $U \in [-c, c]$ reflects the
K\"ahler class of the resulting metric $\tilde{g}$, the volume of
the zero section $E_0$ is $(1-\frac{k}{n} c) V_{FS}$, where $V_{FS}$
denote the volume of $\mathbb{CP}^{n-1}$ endowed with
$Ric(g_{FS})=g_{FS}$, and the volume of the infinity section
$E_{\infty}$ is $(1+\frac{k}{n} c) V_{FS}$. To produce examples in
each of the K\"ahler class of $M_{n,k}$ is equivalent to produce
examples of $\phi(U)$ for any $c \in (0, \frac{n}{k})$ and yet
satisfying inequalities related to $H>0$. We are able to construct
such $\phi(U)$ by establishing some delicate estimates on even
polynomials with large degrees.

Recently, Alvarez-Chaturvedi-Heier \cite{ACH} studied the pinching
constants of Hitchin's examples on $M_{2,k}$.

\begin{theorem} [Alvarez-Chaturvedi-Heier \cite{ACH}] \label{ACH intro}
The local and global pinching constant of holomorphic sectional
curvature are the same for any of the Hitchin's examples on
$M_{2,k}$. The maximum among them is $\frac{1}{(2k+1)^2}$ and the
ray of the corresponding K\"ahler classes is $b[E_{\infty}]-aE_0$ of
the slope $\frac{b}{a}=\frac{2k+2}{2k+1}$.
\end{theorem}

Recall that the local holomorphic pinching constant is the maximum
of all $\lambda \in (0,1]$ such that $0<\lambda H(\pi^{,}) \leq
H(\pi) $ for any $J-$invariant real $2-$planes $\pi, \pi^{,} \subset
T_p(M)$ at any $p \in M$, while the global holomorphic pinching
constant is the maximum of all $\lambda \in (0,1]$ such that there
exists a positive constant $C$ so that $\lambda C \leq H(p,\pi) \leq
C$ holds for any $p \in M$ and any $J$-invariant real 2-plane $\pi
\subset T_p(M)$. Obviously the global holomorphic pinching constant
is no larger than the local one. We show that the conclusion of
Theorem \ref{ACH intro} is not always true for other K\"ahler
metrics with $U(n)$-symmetry and with $H>0$, which again reflects
the specialness of Hitchin's examples.

\begin{proposition}
There exist K\"ahler metrics with $H>0$ on $M_{n,k}$ whose local and
global pinching constants for holomorphic sectional
curvature are not equal.

In general, the local holomorphic pinching constant of any
$U(n)$-invariant K\"ahler metric on $M_{n,k}$ is bounded from above
by $\frac{1}{k^2}$.
\end{proposition}

A direct calculation enables us to generalize the result of
Alvarez-Chaturvedi-Heier \cite{ACH} to $M_{n, k}$. It is interesting
to see that the optimal pinching constant is dimension free. It is
the same constant $\frac{1}{(2k+1)^2}$ discovered in \cite{ACH},
with the corresponding K\"ahler class on $M_{n,k}$ satisfies
\[
b[E_{\infty}]-a[E_0] \ \ \text{where}\ \  b=\frac{2k+2}{2k+1}a>0.
\]
It is unclear to us how to determine the optimal holomorphic
pinching constant among all $U(n)$-invariant K\"ahler metrics on
$M_{n,k}$, though we believe it is achieved among Hitchin's
examples. Motivated by the result of Alvarez-Chaturvedi-Heier
\cite{ACH}, we would like to propose the following question:

\begin{question} \label{pinching intro}
Is the following statement true? If a compact K\"ahler surface with
$H>0$ has its local pinching constant $\lambda > \frac{1}{9}$, then
it must be biholomorphic to $\mathbb{CP}^{2}$ or $\mathbb{CP}^{1}
\times \mathbb{CP}^{1}$.
\end{question}

As a partial evidence on Question \ref{pinching intro}, by using
some previous results on positive orthogonal bisectional curvature,
we give a complete classification of compact K\"ahler manifolds with
local holomorphic pinching constant $\lambda \geq \frac{1}{2}$. In
the case of K\"ahler surfaces, they are biholomorphic to
$\mathbb{P}^{2}$ or $\mathbb{P}^{1} \times \mathbb{P}^{1}$.

It is of course desirable to make further studies on Question
\ref{space of H}. We are able to show that the proof of Theorem
\ref{main in intro} can be used to establish the path-connectedness
of all $U(n)$-invariant K\"ahler metrics of $H>0$ on any Hirzebruch
manifold $M_{n,k}$.

\begin{corollary}\label{U(n) path connected}
The space of all $U(n)$-invariant K\"ahler metrics of $H>0$ on
$M_{n,k}$ is path-connected.
\end{corollary}

At this moment, it is not clear to use how to prove the
path-connectedness without the assumption of $U(n)$-symmetry. In the
meantime it seems impossible to make use of the path constructed in
Corollary \ref{U(n) path connected} and improve the holomorphic
pinching constants. In the other direction, as illustrated in the
work of Chen-Tian \cite{ChenTian1} and \cite{ChenTian2}, one may
wonder if the K\"ahler-Ricci flow can be used to study the space of
all K\"ahler metric of $H>0$ and Question \ref{pinching intro}. To
that end, we calculate the holomorphic sectional curvature for the
K\"ahler-Ricci shrinking soliton on Fano Hirzebuch manifold
$M_{n,k}$ with $n>k$ due to Koiso \cite{Koiso} and Cao \cite{Cao},
and also for the noncompact ones on the total space of $H^{-k}
\rightarrow \mathbb{CP}^{n-1}$ with $n>k$ due to
Feldman-Ilmanen-Knopf \cite{F-I-K}.

\begin{proposition} The Cao-Koiso shrinking soliton on Hirzebruch manifold
$M_{n,k}$ have $H>0$ as the ratio $\frac{n}{k}$ is sufficiently
large. If we fix $k=1$, the first example with $H>0$ is on
$M_{3,1}$; if we fix $k=2$, the first one is on $M_{7,2}$.

In the complete noncompact case, the Feldman-Ilmanen-Knopf shrinking
solitons do not have $H>0$ if $k<n \leq k^2+2k$.
\end{proposition}

In fact we expect that none of Feldman-Ilmanen-Knopf shrinking
solitons will have $H>0$. It would be interesting to know if any
complete K\"ahler-Ricci soliton with $H>0$ must be compact, in view
of the recent work Munteanu-Wang \cite{MW} and the previous work of
Ni \cite{Ni}.

As a corollary to the previous works of Zhu \cite{Zhu},
Weinkove-Song \cite{SW2011}, Fong \cite{Fong}, Guo-Song
\cite{GuoSong} on K\"ahler-Ricci flow on Hirzebruch manifolds with
Calabi's symmetry, we exhibit various pinching behaviors along the
K\"ahler-Ricci flow when the initial metrics are chosen from
examples constructed in Theorem \ref{main in intro}, in particular
we have the following:

\begin{corollary}
$H>0$ is not preserved under the K\"ahler-Ricci flow.
\end{corollary}

The above corollary entails the following question: Can we construct
a suitable one-parameter family of deformation of K\"ahler metrics
$M_{n,k}$ so that the holomorphic pinching constant is monotone
along the deformation?

We would like to point out another generalization of Hitchin's
exmaples. In a very recent work of Alvarez, Heier, and the
second-named author \cite{AHZ}, it was proved that the
projectivization ${\mathbb P}(E)$ of any Hermitian vector bundle $E$
over a compact K\"ahler manifold with $H>0$ also admits a K\"ahler
metric of $H>0$. The resulting metric on ${\mathbb P}(E)$ is of the
form (\ref{Hitchin intro}) for $s$ sufficiently small. Instead of
working with the line bundle $H^{-k}$ on $\mathbb{CP}^{n-1}$, it is
possible to apply the method of Koiso-Sakane developed in the proof
of Theorem \ref{main in intro} to get more examples of K\"ahler
metrics of $H>0$ on some $\mathbb{CP}^{k}$ bundles. For example,
consider $M=\mathbb{CP}^{n_1-1} \times \mathbb{CP}^{n_2}$ and
$L=\pi_1^{\ast} H_1^{-1} \otimes \pi_1^{\ast} H_2^{-k_2}$ where
$H_1$ and $H_2$ are hyperplanes bundles on $\mathbb{CP}^{n_1-1}$ and
$\mathbb{CP}^{n_2}$, $\pi_1$ and $\pi_2$ are projections to its
factors. Then we can produce a $\mathbb{CP}^{n_1}$-bundle over
$\mathbb{CP}^{n_2}$ as a suitable compactification of $L^{\ast}
\rightarrow M$. It is also interesting to study the space of all
K\"ahler metrics of $H>0$ on it.

In complex dimensions higher than one, it is highly desirable to
find examples of compact K\"ahler manifold with $H>0$. Among the
known such examples are all compact Hermitian symmetric spaces and
some K\"ahler $C$-spaces (rational homogeneous space).  In the last
section of the paper, we study the holomorphic pinching constant for
the canonical K\"ahler-Einstein metric on the flag 3-fold, the only
K\"ahler $C$-space which is not Hermitian symmetric in dimension 3.
We also demonstrate a higher dimensional projective manifold such
that some of its classes admit K\"ahler metric with $H>0$ while some
do not. More precisely, we prove:

\begin{theorem} \label{some classes not}
Let $M$ be the hypersurface in $\mathbb{CP}^n \times \mathbb{CP}^n$
defined by $ \sum_{i=1}^{n+1} z_i w_i = 0$ equipped with the
restriction of the product of the Fubini-Study metric, where $z$,
$w$ are the homogeneous coordinates. Then the holomorphic pinching
constant of $M$ is $\frac{1}{4}$.

Consider $N$ which is a smooth bidegree $(p,1)$ hypersurface in
$\mathbb{CP}^r \times \mathbb{CP}^s$ where $r, s \geq 2$, $p \geq
1$, and $p > r+1$, then some K\"ahler classes of $N$ admit K\"ahler
metrics of $H>0$ and some do not.
\end{theorem}

Similarly as in Question \ref{pinching intro},  we may ask

\begin{question} \label{pinching 3fold}
Is it true that if a compact K\"ahler $3$-fold with $H>0$ has its
local holomorphic pinching constant $\lambda > \frac{1}{4}$, then it
must be biholomorphic to a compact Hermitian symmetric space?
\end{question}

In a sequel of this paper, we will study examples of K\"ahler
metrics with $H>0$ on other rational surfaces, other K\"ahler
$C$-spaces, and higher dimensional projective manifolds.

This paper is organized as follows: In Section 2, we prove the
classification theorem of compact K\"ahler manifolds with local
holomorphic pinching constant $\lambda \geq \frac{1}{2}$. In Section
3, we prove the main Theorem \ref{main in intro} and studies the
relation between K\"ahler-Ricci flow and $H>0$ on $M_{n,k}$. In the
last Section 4, we study the canonical K\"ahler-Einstein metric on
the flag $3$-fold and prove Theorem \ref{some classes not}. We end
the paper with some discussions on $H>0$ in the higher dimension
case from the submanifold point of view.

\vspace{0.5cm}

\section{Holormophic sectional curvature: preliminary results}

Let us begin the definition of various curvatures on a K\"ahler
manifold.

\begin{definition}

Let $(M, g, J)$ be a K\"ahler manifold of complex dimension $n
\geq 2$ with the Levi-Civita connection $\nabla$ and Riemannian curvature tensor $R$.

(1) Sectional curvature for any real $2$-plane $\pi \subset T_p(M)$
is defined by $K(\pi)=\frac{R(X, Y, Y, X)}{|X|^2|Y|^2-g(X,Y)^2}$
where $\pi=\operatorname{span}\{X,Y\}$.

(2) Holomorphic sectional curvature (H) for any $J$-invariant real
$2$-plane $\pi \subset T_p(M)$ is defined by $H(\pi)=\frac{R(X, JX,
JX, X)}{|X|^4}$ where $\pi=\operatorname{span}\{X,JX\}$. For the
sake of simplicity, we freely use $H(X)$, $H(X-\sqrt{-1}JX)$ or
$H(\pi)$ for holomorphic sectional curvature.

(3) (Holomorphic) bisectional curvature for any two $J$-invariant
real $2$-planes $\pi, \pi^{,} \subset T_p(M)$ is defined by $B(\pi,
\pi^{,})=\frac{R(X, JX, JY, Y)}{|X|^2|Y|^2}$ where
$\pi=\operatorname{span}\{X,JX\}$ and
$\pi^{,}=\operatorname{span}\{Y,JY\}$.

\end{definition}

In the study of K\"ahler manifolds with positive curvature, it is
useful to consider some pinching condition in either a local or a
global sense.

\begin{definition} [Local pinching and global pinching]\label{pinching def}

Let $\lambda, \delta \in (0,1)$, we define the following pinching
conditions on a K\"ahler manifold $(M, g)$.

(1) $\lambda \leq H \leq 1$ in the local sense if for any $p \in M$,
$0<\lambda H(\pi^{,}) \leq H(\pi) \leq \frac{1}{\lambda} H(\pi^{,})$
for any $J$-invariant real $2$-planes $\pi, \pi^{,} \subset T_p(M)$.
In other words, there exists a function $\varphi(p)>0$ on $M^n$ such
that $0<\lambda \varphi(p) \leq H(p, \pi) \leq \frac{1}{\lambda}
\varphi(p)$ for any $p$ and any holomorphic plane $\pi \subset
T_p(M)$.

(2) $\delta \leq K \leq 1$ in the local sense if for any $p \in M$,
$0<\delta K(\pi^{,}) \leq K(\pi) \leq \frac{1}{\delta} K(\pi^{,})$
for any two real $2$-planes $\pi, \pi^{,} \subset T_p(M)$.

(3) $\lambda \leq H \leq 1$ in the global sense if $\lambda \leq
H(\pi) \leq 1$ for any $p \in M$ and any $J$-invariant real
2-plane $\pi \subset T_p(M)$. $\delta \leq K \leq 1$ in the global
sense is defined similarly.

\end{definition}

K\"ahler manifolds with $H>0$ are less understood and somewhat
mysterious. For example, if one works with linear algebra aspects of
curvature tensors, then $H>0$ alone does not give any helpful
information on the Ricci curvature. In fact, most of the Hirzebruch surfaces in Hitchin's examples are not Fano, thus
do not admit any K\"ahler metric with positive Ricci curvature.
Nonetheless one may study K\"ahler manifolds with $H>0$ pinched by
a large constant. In this regard, the following results of Berger
\cite{Berger1960} and Bishop-Goldberg \cite{BG1963} are very interesting.

\begin{proposition}[Berger \cite{Berger1960}]\label{Berger1960}
Let $(M^{n},g)$ be K\"ahler, then $0<\lambda \leq H \leq 1$ in the
local sense implies $\frac{7\lambda-5}{8} \leq K \leq
\frac{4-\lambda}{3}$ in the local sense.
\end{proposition}

\begin{proposition}[Bishop-Goldberg \cite{BG1963}] \label{BG}
If $(M^{n},g)$ is K\"ahler, then $0<\lambda \leq H(p) \leq 1$
implies \[ \frac{1}{4}[3(1+\cos^2\theta) \lambda-2] \leq K(X,Y) \leq
1-\frac{3}{4} \lambda \sin^2\theta \] for any unit tangent vectors
$X, Y$ at $p$ with $g(X,Y)=0$ and $g(X,JY)=\cos \theta$. In
particular, $\lambda$-holomorphic pinching implies
$\frac{1}{4}(3\lambda-2)$-pinching on sectional curvatures.
\end{proposition}

In the proof of the above Proposition \ref{Berger1960}, Berger
discovered an interesting inequality.
\begin{lemma}[Berger \cite{Berger1960-1} and \cite{Berger1960}]
\label{Berger lemma} Let $(M^n, g)$ be a K\"ahler manifold and
$0<\lambda \leq H \leq 1$ in the local sense, then for any unit
vector $X, Y$ with $g(X,Y)=0$ and $g(X,JY)=\cos {\theta}$, we have
\begin{equation}
\lambda-\frac{1}{2}+\frac{\lambda}{2} \cos^2 \theta \leq R(X, JX,
JY,Y) \leq 1-\frac{\lambda}{2}+\frac{1}{2}\cos^2 \theta.
\label{Berger pinching}
\end{equation}
\end{lemma}

For the convenience of the readers, we sketch Berger's proof of
Lemma \ref{Berger lemma}, as it will be crucial in the proof of
Proposition \ref{half pinching} below.

\begin{proof}[Berger's proof of Lemma \ref{Berger lemma}]
Given any unit vector $X, Y$ with $g(X,Y)=0$ and $g(X,JY)=\cos
{\theta}$, consider
\begin{equation}
\lambda \leq \frac{1}{2}\Big[H(aX+bY)+H(aX-bY)\Big] \leq 1
\label{Bergerlemma1}
\end{equation}
By the left half of inequality (\ref{Bergerlemma1}), we
conclude that
\begin{equation}
(H(X)-\lambda) a^4+(R(X, JX, JY, Y)+2R(X, JY, JY, X)-\lambda) 2a^2
b^2+(H(Y)-\lambda) b^4 \geq 0  \label{Bergerlemma2}
\end{equation}
holds for any real numbers $a$, $b$.
Apply $H(X), H(Y) \leq 1$, it follows from (\ref{Bergerlemma2}) that
\begin{equation}
R(X, JX, JY, Y)+2R(X, JY, JY, X) \geq 2\lambda-1.
\label{Bergerlemma3}
\end{equation}
Next consider
\begin{equation}
\lambda \leq \frac{1}{2}\Big[H(aX+bJY)+H(aX-bJY)\Big] \leq 1.
\label{Bergerlemma4}
\end{equation}
Since  $$
\frac{1}{2} \{  (a^2+b^2+2ab \cos \theta)^2 +  (a^2+b^2-2ab \cos \theta)^2 \} =
(a^2+b^2)^2+4a^2b^2 \cos^2 \theta,$$
a similar argument as in
(\ref{Bergerlemma2}) and (\ref{Bergerlemma3}) leads to
\begin{equation}
3R(X, JX, JY, Y)-2R(X, JY, JY, X) \geq 2\lambda+2\lambda \cos^2
\theta-1. \label{Bergerlemma5}
\end{equation}
By adding (\ref{Bergerlemma3}) and (\ref{Bergerlemma5}) we have
\begin{equation}
R(X, JX, JY, Y)\geq \lambda+\frac{\lambda}{2} \cos^2
\theta-\frac{1}{2}. \label{Bergerlemma6}
\end{equation}
The right half of inequality (\ref{Berger pinching}) can be proved
similarly if we work on the right halves of inequalities in both
(\ref{Bergerlemma1}) and (\ref{Bergerlemma4}).
\end{proof}

It is possible to get some characterization of K\"ahler manifolds
with a large holomorphic pinching constant $\lambda$. For example,
Bishop-Goldberg \cite {BG1963} proved that if $\frac{4}{5}<\lambda
\leq H \leq 1$ holds in the local sense on a compact K\"ahler
manifold $(M, g)$, then $M$ has the homotopy type of
$\mathbb{CP}^n$. They also proved in \cite{BG1965} that $\lambda>
\frac{1}{2}$ implies $b_2(M)=1$. Note that a direct calculation
shows that $\mathbb{CP}^k \times \mathbb{CP}^l$ with the product of
Fubini-Study metric has exactly $\frac{1}{2} \leq H \leq 1$ (see
\cite{ACH} for a general result on holomorphic pinching of product
metrics). In light of these results, it is natural to ask if
$\frac{1}{2}<\lambda \leq H \leq 1$ in the local sense implies that
$M^n$ is biholomorphic to $\mathbb{CP}^n$. This is indeed the case
and we have the following:

\begin{proposition} \label{half pinching}
Let $(M^n,g)$ be a compact K\"ahler manifold with $0<\lambda \leq H
\leq 1$ in the local sense, then the following holds:

(1) If $\lambda>\frac{1}{2}$, then $M^n$ is bibolomorphic to
$\mathbb{CP}^{n}$.

(2) If $\lambda=\frac{1}{2}$, then $M^n$ is one of the following

(2a) $M^n$ is biholomorphic to $\mathbb{CP}^{n}$.

(2b) $M^n$ is holomorphically isometric to $\mathbb{CP}^k \times
\mathbb{CP}^{n-k}$ with a product of Fubini-Study metrics.

(2c) $M^n$ is holomorphically isometric to an irreducible compact
Hermitian symmeric space of rank $2$ with its canonical
K\"ahler-Einstein metric.
\end{proposition}

\begin{proof}[Proof of Proposition \ref{half pinching}]
Let us consider $n \geq 2$, the crucial observation is that
$\frac{1}{2} \leq \lambda \leq 1$ in the local sense implies that
$(M^n,g)$ has nonnegative orthogonal holomorphic bisectional
curvature. Namely for any two $J$-invariant planes
$\pi=\operatorname{span}\{X,JX\}$ and
$\pi^{,}=\operatorname{span}\{Y,JY\}$ in $T_p (M)$ which are
orthogonal in the sense that $g(X, Y)=g(X, JY)=0$, then
\[
R(X, JX, JY, Y) \geq 0.
\]
This follows from Berger's inequality (\ref{Berger pinching}).

Nonnegative and positive orthogonal bisectional curvature is well
studied in \cite{ChenX}, \cite{GuZhang}, and \cite{Wilking}.

If $\lambda>\frac{1}{2}$ then $(M^n, g)$ has positive orthogonal
bisectional curvature, it is proved in \cite{ChenX}, \cite{GuZhang},
and \cite{Wilking} that the K\"ahler-Ricci flow evolves such a
metric to positive bisectional curvature, which is biholomorphic to
$\mathbb{CP}^n$ by \cite{Mori} and Siu-Yau \cite{SiuYau}.

If $\lambda>\frac{1}{2}$ then $(M^n, g)$ has nonnegative orthogonal
bisectional curvature, according to a classification result due to
Gu-Zhang (Theorem in 1.3 in \cite{GuZhang}), combining the fact
$H>0$, then the universal covering manifold $(\tilde{M}, \tilde{g})$
is holomorphically isometric to
\begin{equation}
(\mathbb{CP}^{k_1}, g_{k_1}) \times \cdots \times
(\mathbb{CP}^{k_r}, g_{k_r}) \times (N^{l_1}, h_{l_1}) \times \cdots
(N^{k_r}, h_{l_s})  \label{product list}
\end{equation}
Where each of $(N^{l_i}, h_{l_i})$ is a compact irreducible
Hermitian symmetric spaces of rank $\geq 2$ with its canonical
K\"ahler-Einstein metric, The holomorphic pinching constant of such
a metric was well-studied and it is exactly the reciprocal of its
rank, see for example \cite{Chen1977}. On the other hand, the
pinching of product K\"ahler metrics was studied by \cite{ACH}, the
proved that the pinching constant of a produce K\"ahler metric $(M_1
\times M_2, g_1 \times g_2)$ is $\frac{\lambda_1
\lambda_2}{\lambda_1+\lambda_2}$ where $\lambda_1$ and $\lambda_1$
are pinching constants of $M_1$ and $M_2$ respectively. It is
clearly that $\frac{\lambda_1
\lambda_2}{\lambda_1+\lambda_2}=\frac{1}{2}$ is equivalent to
$\lambda_1=\lambda_2=1$.

Therefore if $\lambda=\frac{1}{2}$, The decomposition (\ref{product
list}) reduces to either a single $\mathbb{CP}^{n}$ or a single
irreducible Hermitian symmetric space of rank $2$, or a product of
$\mathbb{CP}^k \times \mathbb{CP}^{n-k}$ with product Fubini-Study
metrics. Obviously this decomposition descends to the original
manifold $(M^n , g)$.

\end{proof}

A natural question following Proposition \ref{half pinching} is what
is the next threshold, if any, for the holomorphic pinching
constants for K\"ahler manifolds with $H>0$. In general, the
situation might be complicated. Note that the canonical
K\"ahler-Einstein metric on a compact Hermitian symmetric space has
holomorphic pinching constant determined by its rank
(\cite{Chen1977}). The K\"ahler-Einstein metrics on a lot of the
K\"ahler $C$-spaces also have $H>0$, and in general one has to work
with the corresponding Lie algebra carefully to determine its
holomorphic pinching constant. Nonetheless, in this paper we focus
on the case of dimension $2$ and $3$, we will see in Section 3 and 4
that Hirzebruch surfaces and the flag 3-space might be the right
objects to provide the next interesting threshold for the
holomorphic pinching constant.

\vspace{0.8cm}

\section{K\"ahler metrics with $H>0$ on Hirzebruch manifolds}

In this section we first review Hitchin's examples on Hirzebruch
surfaces $M_{2,k}$, then we prove the main Theorem \ref{main in
intro} and study the relation between the K\"ahler-Ricci flow and
$H>0$.

\subsection{A review of Hitchin's construction}

Hitchin \cite{Hitchin} proved that any compact K\"ahler surface with
positive sectional curvature is rational. Any rational surface can
be obtained by blowing up points on $\mathbb{P}^{2}$,
$\mathbb{P}^{1} \times \mathbb{P}^{1}$, and Hirzebruch surfaces
$M_{2, k}$. The natural question is which rational surface admits
K\"ahler metric with $H>0$. In this regard, Hitchin proved that any
Hirzebruch surface $M_{2, k}$ admits a Hodge metric of $H>0$.
Moreover, he proved that the blow up of any compact K\"ahler
manifold with positive scalar curvature admits a K\"ahler metric
with positive scalar curvature when the complex dimension $n \geq
2$. As a corrolary, he showed that any rational surface admits a
K\"ahler metric with positive scalar curvature.

In general, given any Hermitian vector bundle $(E, h) \rightarrow
(M,g)$ where $(M,g)$ is a compact K\"ahler manifold, the Chern
curvature form $\Theta (\mathcal{O}_{{\mathbb P}(E)} (1))$ of
$\mathcal{O}_{{\mathbb P}(E) (1)}$ over ${\mathbb P}(E)$ has the fiber direction
component given by the Fubini-Study form, hence is positive.
Therefore
\begin{equation}
\tilde{\omega}=\pi^{\ast} \omega_g+ s \sqrt{-1} \Theta
(\mathcal{O}_{{\mathbb P}(E)} (1))  \label{Hitchin metric}
\end{equation}
is a well-defined K\"ahler metric on $P(E)$ when $s>0$ is sufficiently
small.

Hitchin \cite{Hitchin} studied K\"ahler metrics of the from
(\ref{Hitchin metric}) on Hirzebruch surfaces $M_{2, k}$, Here we
pick $(E, h)=(H^{k} \oplus 1_{\mathbb{CP}^{1}}, h)$ and $(M, g)$
as $(\mathbb{CP}^1, g_{FS})$ where $g_{FS}$ is the standard
Fubini-Study metric and $h$ the induced metric. If we use the local
parametrization $(z_1, (dz_1)^{-\frac{k}{2}}, z_2)$ and write down
the metric locally
\begin{equation}
\tilde{\omega}=\sqrt{-1} \partial {\bar{\partial}}\log (1+|z_1|^2)+
s\sqrt{-1} \partial {\bar{\partial}}\log [(1+|z_1|^2)^k+|z_2|^2]
\end{equation}
In this case, since the vector bundle $H^{k} \oplus
1_{\mathbb{CP}^{1}}$ has nonnegative curvature, the component of the
Chern curvature form $\Theta (\mathcal{O}_{{\mathbb P}(E)} (1))$ along the
base direction is nonnegative, so $\tilde{\omega}$ is in fact
a K\"ahler metric for all $s>0$.

Hitchin \cite{Hitchin} proved that $\tilde{\omega}$ has $H>0$ if
$0<s(1+ks)^2 <\frac{1}{k(2k-1)}$. Alvarez-Chaturvedi-Heier
\cite{ACH} further proved that it suffices to assume
$s<\frac{1}{k^2}$ to guarantee $H>0$. Let us define the optimal
local (global) holomorphic pinching constant to be the maximum value
among all the pinching constants of Hitchin's examples according to
Definition \ref{pinching def}. It was calculated in \cite{ACH} that
the optimal local and global holomorphic pinching constants are the
same and equal to $\frac{1}{(2k+1)^2}$, and the
corresponding $s=\frac{1}{2k^2+k}$. The corresponding K\"ahler class
is $b[E_{\infty}]-aE_0$ where $b=\frac{2k+2}{2k+1}a>0$. In
particular, if $k=1$, then $s=\frac{1}{3}$, the corresponding
K\"ahler metric $\tilde{\omega}$ is not in the anti-canonical class
of $2\pi c_1 (M_{2,1})$, note that $M_{2,1}$ is the only Fano
Hirzebruch surface.

Let us rephrase the question we proposed in Section 1 of this paper.

\begin{question}
Hitchin's examples produce a family of K\"ahler metrics with $H>0$
whose K\"ahler classes only stay in a subset of the K\"ahler cone.
The path does not approach both sides of the essential boundary of
the K\"ahler cone of $M_{2,k}$. Here by essential we mean that here
the vertex of the cone is not counted as the boundary.

Are there K\"ahler metric with $H>0$ from each of the K\"ahler
classes of $M_{2,k}$? In particular, since $c_1 (M_{2,1})>0$, it would be
interesting to know there is any metric with $H>0$ from the anti-canonical
class of $M_{2,1}$.

What is the best holomorphic pinching constant $\lambda_k$ among all
K\"ahler metrics of $H>0$ on the Hirzebruch surfaces $M_{2, k}$?
Note that Hitchin's examples are of $U(2)$-symmetry, it seems
reasonable to expect the optimal holomorphic pinching constant
$\lambda_k$ to be realized by some K\"ahler metric with a large
symmetry.

Let $\lambda_k$ denote the optimal holomorphic pinching constant
among all K\"ahler metrics of $H>0$ on $M_{2, k}$. Is it true that
any compact K\"ahler surface with pinching constant strictly greater
$\lambda_1$ must be biholomorphic to $\mathbb{CP}^{2}$ or
$\mathbb{CP}^{1} \times \mathbb{CP}^{1}$?
\end{question}

\vspace{0.3cm}

\subsection{Hirzebruch manifolds by Calabi's ansatz}

Let us recall a powerful method to construct canonical metrics
pioneered by Calabi (Calabi's ansatz). Our exposition follows more
closely from Koiso-Sakane \cite{KS1986}. As we shall see later, for
Hirzebruch manifolds $M_{n, k}$, Calabi's ansatz can be applied
to produce $U(n)-$invariant K\"ahler metrics which include Hitchin's
examples as a special case.

\subsubsection{K\"{a}hler metrics on  $\mathbb{C}^{\ast}$-bundles
reviewed}

First we review some facts on the construction of a K\"{a}hler
metric on a $\mathbb{C}^{\ast}$-bundle over a compact K\"{a}hler
manifold where $\mathbb{C}^{\ast}=\mathbb{C}-\{0\}$. Given a
holomorphic line bundle $L \rightarrow M$ on a complex manifold $M$,
where $\pi$ is the natural projection, we consider the
$\mathbb{C}^{\ast}$-action on $L^{\ast}=L\setminus {L_0}$, where
$L_0$ is the zero section of $L$. Denote by $H$ and $S$ the two
holomorphic vector fields generated by the $\mathbb {R}^{+}$ and
$\mathbb{S}^{1}$ action, respectively.

Let $\pi: (L, h) \rightarrow (M, g)$ be a Hermitian line bundle over
a compact K\"{a}hler manifold $(M, g)$. Denote by $\tilde{J}$ the
complex structure on $L$. Assume $t$ is a smooth function on $L$
depending only on the norm and is increasing in the norm of
Hermitian metric $h$. Consider a Hermitian metric on $L^{\ast}$ of
the form
\begin{equation}
\tilde{g}=\pi^{\ast}g_{t}+dt^{2}+(dt\circ \tilde{J})^{2}, \label{def
of g}
\end{equation}
where $g_t$ is a family of Riemannian metrics on $M$. Denote by
$u(t)^{2}=\tilde{g}(H,H)$. It can be checked that $u$ depends only
on $t$.

The following results were proved in \cite{KS1986}.

\begin{lemma}[\cite{KS1986}]

The Hermitian metric $\tilde{g}$ defined by (\ref{def of g}) is
K\"{a}hler on $L^{\ast}$ if and only if each $g_{t}$ is K\"{a}hler
on $M$, and $g_{t}=g_{0}-U \Theta(L)$, where $U=\int_{0}^{t} u(\tau )d\tau $, and if the
range of $t$ includes $0$, then the corresponding value of $U$ is $0$.

\end{lemma}

\begin{assumption} \label{assump 1}
We further assume the eigenvalues of the curvature $\Theta(L)$  with respect to $g_{0}$ are constant on M.
\end{assumption}

Let $z_{1}\cdots z_{n-1}$ be local holomorphic coordinates on $M$
and $z_{0}\cdots z_{n-1}$ be local coordinates on $L^{\ast}$ such
that $\frac{\partial}{\partial z_{0}}=H-\sqrt{-1}S$.

\begin{lemma}[\cite{KS1986}]
$\tilde{g}_{0\bar 0}=2u^2,\
 \tilde{g}_{\alpha \bar 0}=2u\partial_{\alpha} t,\ \tilde{g}_{\alpha
\bar{\beta}}=g_{t \alpha \bar{\beta}}+2\partial_{\alpha}
t\partial_{\bar{\beta}} t.$ Define $p=det(g_{0}^{-1}\cdot g_{t})$,
then $det(\tilde{g})=2u^2 \cdot p \cdot det(g_{0})$.
\end{lemma}

\begin{lemma}[\cite{KS1986}]\label{Koiso local}
If we assume that $\partial_{\alpha} t=\partial_{\bar{\alpha}} t=0\
\ (1 \leqslant \alpha \leqslant n-1)$ on a fiber, and if a function
f  on $L^{\ast}$ depends only on $t$, then
$\partial_{0}\partial_{\bar{0}}f=u\frac{d}{dt}(u\frac{df}{dt}),\
\partial_{\alpha}\partial_{\bar{0}}f=0,\
\partial_{\alpha}\partial_{\bar{\beta}}f=-\frac{1}{2}u\frac{df}{dt} \Theta(L)_{\alpha
\bar{\beta}}$. Moreover, the Ricci curvature of $\tilde{g}$ becomes:
$\tilde{R}_{0\bar 0}=-u\cdot \frac{d}{dt}(u
\cdot\frac{d}{dt}(log(u^{2}p))),\ \tilde{R}_{\alpha \bar 0}=0, \ \
\tilde{R}_{\alpha \bar{\beta}}=R_{0\alpha
\bar{\beta}}+\frac{1}{2}u\cdot \frac{d}{dt}(log(u^{2}p))\cdot
\Theta(L)_{\alpha \beta}$.
\end{lemma}

It is convenient to introduce the new functions $\phi(U)=u^2(t)$ and
$ Q(U)=p$. Recall that $\sqrt{h}$ is the norm of Hermitian metric on
$L$. Since $\frac{dU}{\sqrt{\phi(U)}}=dt$ and $
\frac{dU}{\phi(U)}=\frac{d\sqrt{h}}{\sqrt{h}}$, for any given
$\phi(U)$ we can solve for $t$ with respect to $\sqrt{h}$, hence
recover the metric $\tilde{g}$. The following lemma characterizes
any $\phi(U)$ which corresponds to a well-defined (maybe incomplete)
K\"ahler metric in the form of (\ref{def of g}) on the total space
of $L^{\ast}$.

\begin{lemma} \label{well defined on star}

Given any hermitian line bundle $(L, h)$ over a compact K\"ahler
manifold $(M,g_0)$ and Assumption \ref{assump 1} holds. Fix $-\infty
< U_{min}<U_{max} \leq +\infty$ such that $g_t \doteq g_0-U
\Theta(L)$ remains positive on $(U_{min},U_{max})$.

Let $\phi(U)$ be a smooth positive function on $(U_{min},U_{max})$
with $\phi(U_{min})=\phi(U_{max})=0$. We further assume that
$\int_{U_{min}}^{U} \frac{dU}{\phi(U)}=+\infty$, $\int_{U}^{U_{max}}
\frac{dU}{\phi(U)}=+\infty$ and $\int_{U_{min}}^{U}
\frac{dU}{\sqrt{\phi(U)}}$ is finite for all $U \in
(U_{min},U_{max})$.

Then we can solve for $t$ as a function of $\sqrt{h}$ which is
strictly increasing on $\sqrt{h}$, and $t$ has the range
$(t_{min},t_{max})$ which contains $0$. Therefore we can get a
well-defined smooth K\"{a}hler metric $\tilde{g}$ in the form of
(\ref{def of g}) on $L^{\ast}$.
\end{lemma}


\subsubsection{Metric completion by compactification, K\"ahler
metrics on ${\mathbb P}(L \oplus 1)$.}

Koiso-Sakane \cite{KS1986} had a general discussion on when the
K\"{a}hler metric $\tilde{g}$ on $L^{\ast}$ admits a
compactification so that it can be extended onto ${\mathbb P}(L
\oplus 1)$. We summarize their results below.

\begin{lemma} [Koiso-Sakane \cite{KS1986}] \label{basic compactification}
Let $(t_{min}, t_{max})$ be the range of function $t$ on $L^{\ast}$,
and assume that $t$ extends to ${\mathbb P}(L \oplus 1)$ with the
range $[t_{min}, t_{max}]$, where the subset $M_{min}$ (or
$M_{max}$) defined by $t=t_{min}$ (or $t=t_{min}$) is a complex
submanifold with codimension $D_{min}$ or $D_{max}$. Moreover,
assume the K\"ahler metric $\tilde{g}$ extends to ${\mathbb P}(L
\oplus 1)$, which is also denoted by $\tilde{g}$.

Then it implies that near $U=U_{min}$ the Taylor expansion of
$\phi(U)$ has the first term $2(U-U_{min})$, and near $U=U_{max}$,
it has the first term $2(U_{max}-U)$. In other words, $t-t_{min}$
gives the distance from $M_{min}$ to points in ${\mathbb P}(L \oplus
1)$ and from $M_{max}$ in ${\mathbb P}(L \oplus 1)$, and $t_{max}-t$
the distance from $M_{min}$.
\end{lemma}

A standard example which satisfies the assumptions of Lemma
\ref{basic compactification} is Hirzebruch manifold $M_{n,k}$.
Indeed We may view any $M_{n,k}$ as the compactification of the
total space of $\mathbb{C}^{\ast}$-bundle induced from $k$-th power
of the tautological bundle $H^{-k} \rightarrow \mathbb{CP}^{n-1}$.
Here we assume the base $\mathbb{CP}^{n-1}$ is endowed with the
Fubini-Study metric $Ric(g_0)=g_0$, hence $R_{i \bar{j} k \bar{l}}
(g_0)=\frac{1}{n}(g_{i \bar{j}}g_{k \bar{l}}+g_{k \bar{j}}g_{i
\bar{l}})$. Then from the previous results due to Koiso-Sakane we
get \[ \tilde{g}_{i\overline{j}}=(1+\frac{k}{n}U)g_{i \overline{j}},
\ \ \Theta(H^{-k})=-\frac{k}{n}g_{i\overline{j}},  \ \ t_{i
\overline{j}}=\frac{k}{2n}ug_{i\overline{j}}. \] Here we need to
pick $[U_{min}, U_{max}]$ such that $u_{min}>-\frac{k}{n}$ and
$U_{max}<\infty$.

From now on, we will focus our consideration on $M_{n,k}$. But
before that, let us remark that there are other types of
compactification covered in Lemma \ref{basic compactification}. For
example, consider the tautological bundle $H^{-1} \rightarrow
\mathbb{CP}^{n-1}$, if we pick $U_{min}=-n$ and $U_{max}<\infty$,
the corresponding compactification will produce a K\"ahler metric on
$\mathbb{CP}^n$. Also it is clear that a similar discussion can be
carried out on $C^{\ast}$-bundles obtained from a vector bundle
other than a line bundle.

The following proposition gives the formulas of curvature tensors of
$M_{n,k}$. Note that $U(n)$ acts isometrically on the base
$\mathbb{CP}^{n-1}$, and it can be lifted as isometric actions on
the total space $M_{n,k}$, so the local calculation along a fiber in
Lemma \ref{Koiso local} works on the total space $M_{n,k}$.

\begin{proposition}[Curvature tensors of $\tilde{g}$ on $M_{n,k}$]
\label{curvature tensors} Let us assume that $\partial_{\alpha}
t=\partial_{\overline{\alpha}} t=0\ \ (1 \leqslant \alpha \leqslant
n-1)$ on a fiber. Consider the unitary frame $\{e_0, e_1, \cdots ,
e_{n-1}\}$ on $M_{n,k}$:
\[
e_0=\frac{1}{2\phi} \frac{\partial}{\partial z_0}, \
e_i=\frac{1}{\sqrt{(1+\frac{k}{n}U)g_{i\bar{i}}}}
\frac{\partial}{\partial z_i}\,\, (1\leq i\leq n-1),
\] where $\{e_1, \cdots, e_{n-1}\}$ is a unitary frame on
$(\mathbb{CP}^{n-1},g_0)$. The only nonzero curvature components of
$\tilde{g}$ on $M_{n,k}$ are:
\begin{align*}
A&=\tilde{R}_{0 \bar{0}0 \bar{0}}=-\frac{1}{2} \frac{d^2
\phi}{dU^2},\\
B&=\tilde{R}_{0 \bar{0}i \bar{i}}=
\frac{k^2 \phi-k(n+kU)\frac{d\phi}{dU}}{2(n+kU)^2},\\
C&=\tilde{R}_{i \bar{i}i \bar{i}}=2\tilde{R}_{i \bar{i}j
\bar{j}}=\frac{[2(n+kU)-k^2\phi]}{(n+kU)^2}.
\end{align*}
where $1\leq i, j\leq n-1$ and $i\neq j$.
\end{proposition}

Proposition \ref{curvature tensors} leads to the following
characterization of $U(n)$-invariant K\"ahler metrics with $H>0$ on
$M_{n,k}$.

\begin{proposition}[The generating function $\phi$]
Any $U(n)$-invariant K\"ahler metric on $M_{n,k}$ has positive
holomorphic sectional curvature if and only if
\begin{equation}
A>0, C>0, 2B>-\sqrt{AC}.
\end{equation}
In other words, it is characterized by a smooth concave function
$\phi(U)$ where $-\infty<U_{min} \leq U \leq U_{max}<+\infty$ with
$1+\frac{k}{n}U_{min}>0$ such that the following conditions hold:

(1)  $\phi>0$ on $(U_{min}, U_{max})$,
$\phi(U_{min})=\phi(U_{max})=0$, $\phi'(U_{min})=2$, and
$\phi'(U_{max})=-2$.

(2) \[ \phi(U)<\frac{2}{k^2}(n+kU), \ \
\frac{k\phi}{n+kU}-\phi'>-\sqrt{\phi^{\prime\prime}
\big[\frac{\phi}{2}-\frac{1}{k^2}(n+kU)\Big]}
\] for any $U \in [U_{min},U_{max}]$.
\end{proposition}

As a corollary of Proposition \ref{curvature tensors}, we also have
the following rough estimates on holomorphic pinching constants of
$U(n)$-invariant K\"ahler metrics on $M_{n,k}$.

\begin{corollary}[A rough estimate on holomorphic pinching constants]
Any $U(n)$-invariant K\"ahler metric on $M_{n,k}$ with $H>0$ have
its local holomorphic pinching constant bounded from above by
$\frac{1}{k^2}$, Fix $0<c<\frac{n}{k+2}$, if we consider any
$U(n)$-invariant K\"ahler metric with $H>0$ whose corresponding
K\"ahler class lies in the following ray $S$ in the K\"ahler cone
\[
S=\{\ b[E_{\infty}]-a[E_0] \ |\ \text{where}\ \
a=b\,\frac{n-kc}{n+kc} \ \}.
\]
Then its holomorphic pinching constant is bounded from above by
$\frac{2c}{n-kc}$.
\end{corollary}

Another consequence of Proposition \ref{curvature tensors} is the
path-connectedness of K\"ahler metrics of $H>0$ in the same K\"ahler
class on $M_{n,k}$.

\begin{corollary}[A convexity property in a fixed K\"ahler class]
\label{convexity} If $\phi_1$ and $\phi_2$ are two generating
functions of two K\"ahler metrics of $H>0$ in the same K\"ahler
class on $M_{n,k}$, so is any convex combination
$t\phi_1+(1-t)\phi_2$ with $0<t<1$.
\end{corollary}

\subsubsection{Hitchin's examples reformulated}

Hitchin's construction gives a family of K\"ahler metrics with
$U(2)$ symmetry on $M_{2,k}$. We observe a similar construction
works for $M_{n,k}$.

\begin{example}[Hitchin's examples on $M_{n,k}$]
Given $s>0$, $U_{min}=0$, $U_{max}=ns$, define
$\phi_{s}(U)=-\frac{2}{ns} U^2+2U$. Now \[ A=\frac{2}{ns},\,\,
B=\frac{\frac{k^2}{ns}U^2+\frac{4k}{s}U-kn}{(n+kU)^2},\,\,
C=\frac{\frac{2k^2}{ns}U^2+(2k-2k^2)U+2n}{(n+kU)^2}.\]
\end{example}
For all $s>0$, $\phi_s$ gives a family of K\"ahler metric on
$M_{n,k}$. Now assume $n=2$, Hitchin proved that these metrics have
$H>0$ if $s$ suitably small. Indeed it is observed by
Alvarez-Chaturvedi-Heier \cite{ACH} that it suffices to assume
$0<s<\frac{1}{k^2}$ to get $H>0$. They also calculated the pinching
constants of these metrics and concluded that the optimal value is
$\frac{1}{(2k+1)^2}$ when $s=\frac{1}{2k^2+k}$. When $k=1$, the
optimal metric lies in the K\"ahler class
$\frac{4}{3}[D_{\infty}]-[D_0]$.

In fact, we observe that Hitchin's example is canonical in the
following sense.

\begin{proposition} Hitchin's examples can be uniquely characterized
as $U(n)$-invariant K\"ahler metrics on $M_{n,k}$ with the constant
radial curvature $A$. In particular, the following example gives the
unique form of $\phi(U)$ up to a scaling and a translation of
$[U_{min}, U_{max}]$.
\end{proposition}

\begin{example}[Hitchin's example in a canonical form]
\label{other class}

Let $c>0$, $U_{min}=-c$, $U_{max}=c$, define
$\phi_c(x)=c-\frac{x^2}{c}$ on $[-c,c]$. Since
$\phi^{\prime}(-c)=2$, $\phi^{\prime}(c)=-2$, and $\phi(\pm c)=0$,
we have a K\"ahler metric on $M_{2,1}$. Now \[ A=\frac{1}{c},\,\,
B=\frac{\frac{1}{c}U^2+\frac{4}{c}U+c}{2(U+2)^2},\,\,
C=\frac{\frac{1}{c}U^2+2U+4-c}{(U+2)^2}.\]
\end{example}

If we assume  $0<c<2$, then obviously  $1-\frac{1}{2}U>0$, $A>0$, and $C>0$ on $[-c,c]$. Consider
\[
D \doteq
2B+\sqrt{AC}=\frac{\frac{1}{c}U^2+\frac{4}{c}U+c+\frac{1}{c}
\sqrt{U^2+2cU+c(4-c)}\, \cdot (U+2)}{(U+2)^2}.
\]
Then $D(-c)>0$ is equivalent to $c<\frac{2}{3}$. Moreover, one can check
that the numerator of $D(U)$ is increasing on $U \in (-c,c)$, hence
$D(U)>0$ for any $-c<U<c$ and $0<c<\frac{2}{3}$. Therefore $\phi_c$
provides a family of K\"ahler metrics of $H>0$.

Next let us find the pinching constant for
$\phi_{c}(x)=c-\frac{x^2}{c}$. For any given $\phi_{c}$, the expression of
the holomorphic sectional curvature is
\[
H(X)=A|x_1|^4+4B|x_1 x_2|^2+C |x_2|^4,
\] where $X=x_1 e_1+x_2 e_2$ with $|x_1|^2+|x_2|^2=1$.

If we set $t=|x_1|^2$, then $H(X)=(A+C-4B)t^2+t(4B-2C)+C$ with $t\in
[0,1]$. it is elementary to discuss its extremal values. In
particular, we will show that for any $c \in (0,\frac{2}{3})$, the
pinching constant $\inf_{U \in (-c,c)} \frac{\min H}{\max H}(U)$ is
always attained at $U=-c$, i.e. along the zero section of $M_{2,1}$.

Indeed
\[
\min_{||v||=1} H(U,v)=\frac{AC-4B^2}{(A+C-4B)}, \max_{||v||=1} H(U,
v)=A.
\]
Therefore, the local pinching constant equals
\[
\frac{AC-4B^2}{A(A+C-4B)}=\frac{2U^3+(6-3c)U^2-12cU-c^3-2c^2-8c}
{(2U-2-3c)(U+2)^2}.
\]

It is direct to check that the above expression is increasing on $U
\in [-c,c]$. If $U=-c$, it becomes $\frac{2c(c-3c)}{(2-c)(5c+2)^2}$,
When $c=\frac{2}{7} \approx 0.2857$, it attains the maximum
$\frac{1}{9}$. The optimal pinching constant among the family
$\phi_{c}$ agrees with the result of Alvarez-Chaturvedi-Heier
\cite{ACH}, and the corresponding optimal K\"ahler metrics are just
multiples of those in \cite{ACH}.


It is also straightforward to solve the optimal holomorphic pinching
constant of Hitchin's examples on any Hirzebruch manifold $M_{n,k}$.
Given any $1 \leq k<n$, pick $c>0$, $U_{min}=-c$, $U_{max}=c$,
define $\phi_{c}(U)=c-\frac{x^2}{c}$ on $[-c,c]$. We claim that $\phi_c$ gives a K\"ahler metric on $M_{n,k}$ with
$H>0$ as long as $c < \frac{n}{k(2k+1)}$. Note that
\[ A=\frac{1}{c},\,\,
B=\frac{\frac{k^2}{c}U^2+\frac{2nk}{c}U+c k^2}{2(n+kU)^2},\,\,
C=\frac{\frac{k^2}{c}U^2+2kU+2n-c k^2}{(n+kU)^2}.\]

Similarly
\[
D \doteq 2B+\sqrt{AC}=\frac{\frac{k^2}{c}U^2+\frac{2nk}{c}U+ck^2+
\frac{n+kU}{c}\sqrt{k^2U+2kcU+(2n-ck^2)c}}{(n+kU)^2}
\]

First note that $0<c < \frac{n}{k(2k+1)}$ is equivalent to
$D(-c)>0$, then similarly we could show under this condition on $c$
the numerator of $D$ is strictly increasing on $U \in (-c, c)$,
therefore $D(U)>0$ holds on $[-c,c]$.

Note that when $n\geq 4$, for any positive integer $k$ satisfying $k(2k+1) < n$,  we may pick $c=1$. In this case the K\"ahler class of the resulting metric is
proportional to the anti-canonical class.

For the above metric on $M_{n,k}$ given by $\phi_c (x) = c - \frac{x^2}{c}$ on $[-c,c]$, where $0<c<\frac{n}{k(2k+1)}$ is a constant, one can carry out a similar calculation to
conclude that maximal local pinching constant achieves its maximum
at $U=-c$, which is
\[
\frac{2c(n-c(2k^2+k))}{(n-ck)((3k+2)c+n)}.
\]
It can be shown that it obtains its maximum value at
$c=\frac{n}{4k^2+3k}$, and the optimal pinching constant is
$\frac{1}{(2k+1)^2}$. Note that the optimal pinching constant is
dimension free.

\vspace{0.3cm}

\subsection{New examples and the proof of Theorem \ref{main in intro}}

Next let us consider $M_{2,1}$ which is the only Fano Hirzebruch
surface. Note that K\"ahler metrics of Hitchin's examples on
$M_{2,1}$ can not be proportional to the anti-cananical class. A
natural question is whether there exists a K\"ahler metric with
$H>0$ in $2\pi c_1(M_{2,1})$. Note that the corresponding $\phi(U)$
of such a metric must satisfy $(2+U_{max})=3(2+U_{min})$ besides
$\phi(U_{min})=\phi(U_{max})=0$, $\phi'(U_{min})=2$, and
$\phi'(U_{max})=-2$ required by the smooth compactification.

In the following we exhibit such an example with different global
and local holomorphic pinching constants.

\begin{proposition}
[A new family of K\"ahler metrics of $H>0$ on $M_{2,1}$]
\label{canonical class}

Given any real number $0<c<\frac{6}{5}$, pick a real number $\mu \in
(\frac{1}{2}c, c)$, define $\phi_{c, \mu}: [-c, c] \rightarrow
\mathbb{R}$ by
\[
\phi_{c,\mu}(U)=\mu-(\frac{1}{c^3}-\frac{\mu}{c^4})
U^4-(\frac{2\mu}{c^2}-\frac{1}{c})U^2
\]
Thus $\phi_{c, \mu}$ determines a family of K\"ahler
metrics on $M_{2,1}$, and in particular when $c=1$, the K\"ahler class
of $\phi_{c, \mu}$ is proportional to the anti-canonical class of
$M_{2,1}$.

There exists some $\delta \in (0,\frac{1}{2})$ which depends on $c$
such that for any $\frac{1}{2}c < \mu <(\frac {1}{2}+\delta)c$,
$\phi_{c,\mu}(U)$ defines a K\"ahler metric on $M_{2,1}$ with $H>0$.
\end{proposition}

\begin{proof}[Proof of Proposition \ref{canonical class}]

We begin with the curvature tensors of $\phi_{c, \mu}(U)$:
\begin{align*}
&A=6(\frac{1}{c^3}-\frac{\mu}{c^4})U^2+(\frac{2\mu}{c^2}-\frac{1}{c}), \\
&B=\frac{3(\frac{1}{c^3}-\frac{\mu}{c^4})U^4+
8(\frac{1}{c^3}-\frac{\mu}{c^4})U^3+
(\frac{2\mu}{c^2}-\frac{1}{c})U^2+
4(\frac{2\mu}{c^2}-\frac{1}{c})U+\mu}{2(U+2)^2},\\
&C=\frac{(\frac{1}{c^3}-\frac{\mu}{c^4})U^4+
(\frac{2\mu}{c^2}-\frac{1}{c})U^2+2(U+2)-\mu}{(U+2)^2}.
\end{align*}

First note that for any $c \in (0,2)$, $\mu \in (\frac{1}{2}c,c)$,
we have $A>0$ for any $U \in [-c, c]$, then since
$C(-c)=\frac{2}{2-c}$ and $C(U+2)^2$ is increasing on $[-c,c]$, we
also have $C>0$.

Next one can check that
\[
2B+\sqrt{AC}|_{U=-c}=-\frac{2}{2-c}+\frac{1}{c}\sqrt{\frac{2(5c-4\mu)}{2-c}},
\]

From here, it is direct to see that given any $c \in (0,
\frac{6}{5})$, there exists some $\delta>0$ such that for any $
\frac{1}{2}c < \mu< (\frac{1}{2}+\delta)c$, $2B+\sqrt{AC}>0$ at
$U=-c$.

From now on let us consider $T_{\mu}(U)=(U+2)^2(2B+\sqrt{AC})$, it
suffices to show that $T_{\mu} (U)>0$ for any $U \in [-c,c]$ if
$\mu$ is sufficiently close to $\frac{1}{2}c$. Note that
\[
T_{\mu} (U)=P_{\mu}(U)+(U+2)\sqrt{Q_{\mu}(U)}
\] where
\[
P_{\mu}(U)=3(\frac{1}{c^3}-\frac{\mu}{c^4})U^4
+8(\frac{1}{c^3}-\frac{\mu}{c^4})U^3
+(\frac{2\mu}{c^2}-\frac{1}{c})U^2
+4(\frac{2\mu}{c^2}-\frac{1}{c})U+\mu
\]
and
\begin{align*}
Q_{\mu}(U)= & \ \  6(\frac{1}{c^3}-\frac{\mu}{c^4})^2 U^6
+7(\frac{1}{c^3}-\frac{\mu}{c^4})(\frac{2\mu}{c^2}-\frac{1}{c})U^4
+12(\frac{1}{c^3}-\frac{\mu}{c^4})U^3 \\
&
+\Big[(\frac{2\mu}{c^2}-\frac{1}{c})^2+6(4-\mu)(\frac{1}{c^3}-\frac{\mu}{c^4})\Big]U^2
+2(\frac{2\mu}{c^2}-\frac{1}{c})U
+(-\frac{2}{c^2}\mu^2+\frac{8+c}{c^2} \mu -\frac{4}{c}).
\end{align*}

\vskip 3mm

\begin{claim} \label{claim1}
$T_{\frac{c}{2}}(U) > 0$ on $[-c,c]$ and in particular it has a
positive lower bound at $0$.
\end{claim}

\vskip 3mm

\noindent \emph{Proof of Claim \ref{claim1}.} To see it is true,
note that:
\[
T_{\frac{c}{2}}(U)=(\frac{3}{2c^3}U^4+\frac{4}{c^3}U^3+\frac{c}{2})
+(U+2)\sqrt{\frac{3}{2c^6}U^6+\frac{6}{c^3}U^3+\frac{3(8-c)}{2c^3}U^2}.
\]
It suffices to consider the interval $[-c, 0]$, we will show that
$T_{\frac{c}{2}}(U)$ is strictly concave on $[-c, 0]$, thus it
attains its minimum either at $U=-c$ or at $U=0$, which are both
positive.
\[
\dfrac{d^2 T_{\frac{c}{2}} (U)}{ d
U^2}=\frac{6}{c^3}U(3U+4)+\frac{\frac{9}{2c^{12}}U^3 \cdot
R(U)}{\Big(\sqrt
{\frac{3}{2c^6}U^6+\frac{6}{c^3}U^3+\frac{3(8-c)}{2c^3}U^2}\Big)^3},
\]
where
\begin{align*}
R(U)=& \ 6U^8+6U^7+36 c^3 U^5+(108c^3-9c^4)
U^4\\
&+(80c^3-10c^4)U^3+30c^6 U^2+(108c^6-12c^7) U +c^8-20c^7+96c^6.
\end{align*}

Note that $R(-c)=4c^{6}(2-c)^2>0$, we will prove that $\frac{d R(U)}{d
U}>0$ on $[-c, 0]$, which leads to $R(U)>0$ for any $0<c<\frac{6}{5}$.
Indeed,
\begin{align*}
\frac{1}{6} \frac{d R(U)}{d U}=& \ 8 U^7+7 U^6+30 c^3 U^4+6(12
c^3-c^4)U^3\\
&+5(8c^3-c^4)U^2+10 c^6 U+(18 c^6 - 2c^7)\\
=&I_1+I_2+I_3.
\end{align*} where we have
\begin{align*}
I_1 =& \ 10 c^6 U + \frac{72}{5} c^6-2c^7 \geq \frac{72}{5}
c^6-12 c^7 = 12 (\frac{6}{5}-c) \geq 0, \\
I_2 =& \ 8 U^7+ 7 U^6+ \frac{13}{6} c^3 U^4 \geq U^6
(8U+7+\frac{13}{6} c) \geq 7U^6(1-\frac{5}{6} c) \geq 0,\\
I_3 =& \ \frac{167}{6} c^4 U^4+\frac{18}{5}c^6+6(12
c^3-c^4)U^3+5(8c^3-c^4)U^2,
\end{align*}for any $U \in [-c, 0]$ where $0<c<\frac{6}{5}$.

Next we prove $I_3>0$ on $[-c, 0]$.
\begin{align*}
I_3 \geq & \ U^2 \Big[\frac{167}{6} c^3 U^2+ 6(12 c^3-c^4)
U+5(8c^3-c^4) +\frac{18}{5} c^4 \Big].
\end{align*} Let $S(U)$ denote the quadratic function inside the
bracket:
\[
S(U)=\frac{167}{6} c^3 U^2+ 6(12 c^3-c^4) U+5(8c^3-c^4)
+\frac{18}{5} c^4.
\] Now it is straightforward to see that under the assumption
$0<c<\frac{6}{5}$, $S(U)$ attains its minimum at $U=-c$, and
\[
S(-c)=c^3 (\frac{203}{6} c^2-\frac{367}{5}c+40)>0
\]

Putting these together, we have proved that $\dfrac{d^2
T_{\frac{c}{2}} (U)}{ d U^2}$ is strictly negative on $[-c, 0]$, and
therefore $T_{\frac{c}{2}}(U) > 0$ on $[-c,c]$.

\vskip 3mm

Now let us continue with the proof of Proposition \ref{canonical class}, note that
as $\mu \rightarrow (\frac{c}{2})^+$, $T_{\mu}(U)$ converges to
$T_{\frac{c}{2}} (U)$ uniformly on $[-1, -U_0] \cup [U_0,1]$ for
some fixed small number $U_0>0$,
\begin{align*}
\frac{\partial T_{\mu}(U)}{\partial \mu} = & \ \ \frac{\partial
P}{\partial \mu}+
\frac{(U+2)}{2\sqrt{Q}} \frac{\partial Q}{\partial \mu} \\
=& \ (-\frac{3}{c^4}U^4-\frac{8}{c^4}U^3+\frac{2}{c^2}U^2+\frac{8}{c^2}U+1)\\
&+\frac{(U+2)}{2\sqrt{Q}}
\Big\{-\frac{12}{c^4}(\frac{1}{c^3}-\frac{\mu}{c^4})U^6+
7\Big[\frac{2}{c^2}(\frac{1}{c^3}-\frac{\mu}{c^4})-
\frac{1}{c^4}(\frac{2\mu}{c^2}-\frac{1}{c})\Big]U^4\\
&
-\frac{12}{c^4}U^3-\Big[\frac{4}{c^2}(\frac{2\mu}{c^2}-\frac{1}{c})
-6(\frac{1}{c^3}-\frac{\mu}{c^4})-\frac{6}{c^4}(4-\mu)\Big]U^2
+\frac{4}{c^2}U+ \frac{8+c-4\mu}{c^2} \Big\}
\end{align*}

Take $U_0=\min\{\frac{1}{100}, \frac{1}{100}c^2\}$, we see that
$\frac{\partial T_{\mu}(U)}{\partial {\mu}}$ is strictly positive
for any $|U|<U_0$ as long as $\mu-\frac{c}{2}$ is small enough. In
other words, we can find some $\delta>0$ such that $T_{\mu }(U)>0$
on $[-U_0, U_0]$ for any $\frac{c}{2}<\mu<(\frac{1}{2}+\delta)c$.
Moreover, $T_{\mu}(U)$ converges to $T_{\frac{c}{2}} (U)$ outside
$[-U_0, U_0]$, hence we get $T_{\mu}(U)>0$ for $\frac{1}{2}c < \mu
<(\frac {1}{2}+\delta)c$.
\end{proof}

\begin{example}[Pinching constants in the anti-canonical class]
 Let us now focus on the anti-canonical examples constructed in the
previous proposition, namely, with $c=1$. We expect that Proposition \ref{canonical class} is still true for any $\frac{1}{2} < \mu < \frac{3}{4}$.
Numerical tests suggest that $2B+\sqrt{AC}$ is indeed positive on $U
\in [-1,1]$ for any $\frac{1}{2} < \mu < \frac{3}{4}$. However, it seems
rather tedious to prove it rigorously, as $T_{\mu}(U)$ is not
always increasing on $[-1,1]$. We need some better estimates on the
critical points of $T_{\mu}(U)$ which lie in $[-1,0)$.

The following table shows that if $\mu$ is close to $\frac{1}{2}$,
then the local pinching constants of K\"ahler metric generated by
$\phi_{\mu}$ differs from the global one.
\begin{center}
    \begin{tabular}{ | l | l | l | l | l | l| p{5cm} |}
    \hline
    Intervals of $U$ & $[-1, U_1]$ & $[U_1, U_2]$
    & $[U_2, U_3]$ & $[U_3, U_4]$ & $[U_4, 1]$ \\ \hline
    $\min_{||v||=1} H(U, v)$ & $\frac{AC-4B^2}{A+C-4B}$ &
    $\frac{AC-4B^2}{A+C-4B}$
     & $A$ & $\frac{AC-4B^2}{A+C-4B}$ &$\frac{AC-4B^2}{A+C-4B}$ \\ \hline
    $\max_{||v||=1} H(U, v)$ & $A$ &
    $C$
     & $C$ & $C$ &$A$ \\ \hline
    \end{tabular}
\end{center}
In the above, $U_1<U_4$ are values which corresponds to $A=C$, and
$U_2<U_3$ are values which corresponds to $A=2B$.

For example along the zero section $U=-1$, $\min
H=\frac{AC-B^2}{A+C-4B}=\frac{6-8\mu}{11-4\mu}$ and $\max
H=A(-1)=5-4\mu$. Therefore the pinching constant along zero section
is $\frac{6-8\mu}{(5-4\mu)(11-4\mu)}$, which is close to
$\frac{2}{27}$ as $\mu$ goes close to $\frac{1}{2}$. It is clear
that the global maximum of holomorphic sectional curvature is
attained at $U=-1$ by $A=5-4\mu$ while the global minimum is
attained at $U=0$ by $A=2\mu-1$. Therefore, we conclude that the
local pinching constants of K\"ahler metric generated by
$\phi_{\mu}$ differs from the global one. Indeed, we have
\[
\min_{U \in [-1,1]} \frac{\min_{||v||=1} H(U, v)}{\max_{||v||=1}
H(U, v)}=\frac{4(2\mu-1)}{4-\mu}, \, \, \frac{\min_{U, ||v||=1} H(U,
v)}{\max_{U, ||v||=1} H(U, v)}=\frac{2\mu-1}{5-4\mu}.
\]
\end{example}

Now we are ready to prove our main theorem.

\begin{theorem}\label{main any class}
Given any Hirzebruch manifold $M_{n,k}=\mathbb{P}(H^{k}\oplus
1_{\mathbb{CP}^{n-1}})$, there exists a K\"ahler metric of $H>0$ in
each of its K\"ahler classes.
\end{theorem}

Theorem \ref{main any class} is a corollary of the following
proposition.

\begin{proposition}\label{any class}
Let $n \geq 2$ and $k \geq 1$ be any two integers, there exists some
$p_0(n,k) \in \mathbb{N}$ and a sequence of positive real numbers
$\{\epsilon_p\}_{p \geq p_0}$ with $\lim_{p \rightarrow \infty}
\epsilon_p=0$, such that for any $c \in
(0,\frac{n}{k}-2\epsilon_p]$, there exists $p_1(n ,k, c)>p_0$ with
the following property:

Given any $p \geq p_1$ there exists some $\delta_1>0$ and
$\delta_2>0$ such that for any $\alpha_2 \in (0,\delta_1)$ and
$\mu=\frac{c}{p}+\delta_2$, $\phi(x)$ defined on $[-c, c]$ by
\begin{align}
\phi(x)=\mu-\alpha_2 x^2-\alpha_{2p-2} x^{2p-2}-\alpha_{2p} x^{2p}.
\label{polynomial formula}
\end{align} where
\[
\alpha_{2p-2}=\frac{p\mu-c-(p-1)\alpha_2 c^2}{c^{2p-2}}, \ \ \
\alpha_{2p}=\frac{c-(p-1)\mu+(p-2)\alpha_2 c^2}{c^{2p}},
\] generates a K\"ahler metric with $H>0$ on $M_{n,k}$.
\end{proposition}

\begin{proof}[Proof of Proposition \ref{any class}]
Let $\epsilon_p=\frac{2n}{2p+2k-1}$ and pick any
$c<\frac{n}{k}-2\epsilon_{p}$, we will determine constants
$\delta_1$ and $\delta_2$ step by step. A quick observation is that
in order to make sure both $\alpha_{2p-2}$ and $\alpha_{2p}$
positive, we need
\begin{equation}
\delta_2<\frac{c}{p(p-1)}, \ \text{and} \ (p-1)\delta_1 c^2< p \,
\delta_2.  \label{condition 1}
\end{equation}
Note that
\begin{align}
A=& \ p(2p-1) \alpha_{2p} U^{2p-2}+ (p-1)(2p-3)\alpha_{2p-2} U^{2p-4}
+\alpha_2,  \label{general A}\\
B=& \ \frac{(2p-1)\alpha_{2p}U^{2p}+ 2p \, nk \, \alpha_{2p} \,
U^{2p-1}+ (2p-3)k^2
\alpha_{2p-2} U^{2p-2}}{2(n+kU)^2}  \nonumber\\
&+\frac{(2p-2)nk \alpha_{2p-2} U^{2p-3} +k^2 \alpha_2 U^2
+2nk\alpha_2 U+ k^2\mu}{2(n+kU)^2}, \label{general B}\\
C=& \ \frac{k^2 \alpha_{2p} U^{2p}+k^2 \alpha_{2p-2} U^{2p-2}+k^2
\alpha_2 U^2 + 2kU + 2n-k^2 \mu}{(n+kU)^2}.  \label{general C}
\end{align}

Obviously $A>0$ on $[-c, c]$, note that
\[C>\frac{2n-k^2
\mu-2kc}{(n+kU)^2} \ \ \text{for any} \ U  \in [-c,c]. \] By
plugging $c \leq \frac{n}{k}-2\epsilon_p$ and
$\mu=\frac{c}{p}+\delta_2$ into the above one can conclude that is a
sufficient condition for $C>0$ is
\begin{equation}
\delta_2<\frac{n(2p+2k+1)}{kp(2p+2k-1)}. \label{condition 2}
\end{equation}

Note that $2B+\sqrt{AC}|_{U=-c}>0$ is equivalent to $\phi^{\prime
\prime} (-c)<-\frac{4k^2}{n-kc}$. By a direct calculation we see
that it is further equivalent to
\begin{equation}
\frac{k(2p-1)+2k^2}{c(n-kc)} \big[ \frac{n}{k}-\epsilon_p -c \big]
>\frac{2p(p-1)}{c^2} \delta_2 -\alpha_2 (2p^2-6p+4).
\label{condition 3}
\end{equation}
In order that (\ref{condition 3}) holds, it suffices to pick
$\delta_2$ so that the following holds:
\begin{equation}
\frac{k(2p-1)+2k^2}{(n-kc)} \big[ \frac{n}{k}-\epsilon_p -c \big]
>\frac{2p(p-1)}{c} \delta_2.
\label{condition 4}
\end{equation}
 In other words, for any $c \in (0,\frac{n}{k}-2\epsilon_p]$, it is
easy to pick $\delta_1$ and $\delta_2$ such that all the
inequalities in (\ref{condition 1}), (\ref{condition 2}), and
(\ref{condition 4}) are satisfied.

It remains to show $2B+\sqrt{AC}$ is positive on $[-c,c]$. Motivated
by the proof of Proposition \ref{canonical class} let us introduce
\begin{align*}
T(\mu, \alpha_2, U)=P(\mu, \alpha_2, U)+ (n+kU) \sqrt{Q(\mu,
\alpha_2, U)},
\end{align*} where
\begin{align*}
P(\mu, \alpha_2, U)=2(n+kU)^2 B(U), \ \text{and}\  Q(\mu, \alpha_2,
U)=(n+kU)^2 A(U) \cdot C(U).
\end{align*} where $A, B, C$ are given in (\ref{general A}), (\ref{general
B}), and (\ref{general C}).

We need to show that $T(\mu, \alpha_2, U)>0$ for any $U \in [-c,c]$
under the assumption $c \in (0,\frac{n}{k}-2\epsilon_p]$, $\alpha_2
\in (0,\delta_1)$, and $\mu=\frac{c}{p}+\delta_2$, where $\delta_1$
and $\delta_2$ satisfy (\ref{condition 1}), (\ref{condition 2}), and
(\ref{condition 3}). By a similar argument as in the proof of
Proposition \ref{canonical class}, one checks that for $\alpha_2$,
$\delta_2$ small, there exists $U_0$ sufficiently small.
\[
\dfrac{\partial T(\mu, \alpha_2, U)}{\partial {\mu}}>0, \
\text{and}\  \dfrac{\partial T(\mu, \alpha_2, U)}{\partial
{\alpha_2}}>0
\] for any $|U|<U_0$.

For example, note that
\[
\frac{\partial \alpha_{2p-2}}{\partial \mu}=\frac{p}{c^{2p-2}},
\frac{\partial \alpha_{2p}}{\partial \mu}=-\frac{p-1}{c^{2p}}
\]
Therefore,
\begin{align*}
\frac{\partial P}{\partial \mu} = & \  -\frac{(p-1)(2p-1)k^2}{c^{2p}}
U^{2p} +(2p-3)\frac{k^2p}{c^{2p-2}} U^{2p-2} +
(2p-2) \frac{nkp}{c^{2p-2}} U^{2p-3}+k^2,\\
\frac{\partial Q}{\partial \mu} = & \ \frac{\partial A}{\partial \mu}
C+A\frac{\partial C}{\partial \mu}\\
=& \ \Big( -\frac{p(p-1)(2p-1)}{c^{2p}}
U^{2p}+\frac{p(p-1)(2p-3)}{c^{2p-2}} U^{2p-4} \Big) \cdot C\\
& \ + \Big( -\frac{k^2(p-1)}{c^{2p}} U^{2p}+\frac{k^2 p}{c^{2p-2}}
U^{2p-2} -k^2\Big) \cdot A,\\
\geq & \  \frac{p(p-1)(2p-3)}{2c^{2p-2}} U^{2p-4} (2n-k^2 \mu)
-2k^2\Big(\alpha_2+(p-1)(2p-3)\alpha_{2p-2} U^{2p-4}\Big).
\end{align*}
It follows that when $|U|$, $\alpha_2$, $\delta_2$ are small enough,
\begin{align}
\frac{\partial T(\mu, \alpha_2, U)}{\partial \mu} \  = & \ \
\frac{\partial P}{\partial \mu}+ \frac{(n+kU)}{2\sqrt{Q}}
\frac{\partial Q}{\partial \mu} \ \geq \ \frac{k^2}{2}-\frac{4nk^2
\alpha_2}{2\sqrt{\alpha_2 \cdot n}} \ > \ 0.  \label{derivative
positive}
\end{align}

Therefore to prove Theorem \ref{any class} it suffices to  show that
$T(\frac{c}{p}, 0, U)>0$ for $U \in [-c,c]$. Now we have
\[
\alpha_{2p-2}=0, \ \text{and} \ \alpha_{2p}=\dfrac{1}{p \,
c^{2p-1}}.
\]
Therefore,
\begin{align}
P(\frac{c}{p}, 0, U)=& \ \frac{(2p-1)k^2}{p c^{2p-1}} U^{2p} +
\frac{2nk}{c^{2p-1}} U^{2p-1}+\frac{c}{p}k^2, \\
Q(\frac{c}{p}, 0, U)=& \ \frac{2p-1}{c^{2p-1}} U^{2p-2}
\Big(\frac{k^2}{p c^{2p-1}} U^{2p} + 2 k U + 2n-\frac{c}{p} k^2
\Big).
\end{align}
Let us reparametrize $x=\frac{U}{c}$, then $T(x)$ is defined on
$[-1,1]$.
\begin{align}
T(\frac{c}{p}, 0, x)=& \  P(x)+(n+kcx)\sqrt{Q(x)}\\
=& \ \frac{(2p-1)k^2 c}{p} x^{2p} + 2nk
x^{2p-1}+ \frac{c}{p} k^2 \nonumber\\
& \ +(n+k c \,x)\sqrt{(2p-1)  x^{2p-2} \Big(\frac{k^2}{p} x^{2p} + 2kx
+ \frac{2n} {c}-\frac{k^2}{p} \Big)}
\end{align}

Note that $P(x)$ is increasing on $[-1,0]$ and has a unique zero
$x_0 \in (-1,0)$. We already have $T(-1)>0$ and $T(x_0)>0$ from
(\ref{condition 1}), (\ref{condition 2}), and (\ref{condition 3}).
It suffices to show $T(x)>0$ on $[-1, x_0]$. To that end, let us
introduce
\[
W(x)=-(P(x))^2+(n+kxc)^2 Q(x).
\]

Obviously we also have $W(-1)>0$ and $W(x_0)>0$. By
Lemma \ref{T positive} below we conclude that $W(x)>0$ on $[-1,x_0]$,
which implies $T(x)>0$ for any $x \in [-1, 0]$, thus completing the proof of Proposition \ref{any class}.
\qed

\begin{lemma}\label{T positive}
Given any $n \geq 2$ and $k$ positive integers and $c \in (0,
\frac{n}{k}-2\epsilon_p]$ for $p \geq p_0(n,k)$, there exists
$p_1(n, k, c)$ such that for any $p>p_1$, either there exists
$-1<x_1<x_0$ such that $W(x)$ is increasing on $[-1, x_1]$ and
decreasing on $[x_1, x_0]$,  or $W(x)$ is increasing on $[-1, x_0]$.
\end{lemma}

\noindent \emph{Proof of Lemma \ref{T positive}.}
A straightforward calculation shows that
\begin{align*}
\frac{d W}{d x}&  = \ -2P(x)P^{\prime}(x)+2(n+kcx) kc Q(x)+(n+kcx)^2
Q^{\prime}(x)\\
&  = \ (2p-1)(n+kcx)x^{2p-3} J(x),
\end{align*} where
\begin{align*}
J(x)=& \ -4k^3 c \frac{p-1}{p} x^{2p+1} -2nk^2 \frac{2p+1}{p} x^{2p}
+2k^2 c (2p+1) x^2\\
& \ +(8npk-2kn-\frac{4k^3c}{p}-2k^3c)x+(\frac{4pn^2}{c}-\frac{4n^2}{c}
-2k^2n+\frac{2nk^2}{p}).
\end{align*}

We add a brief remark on $J(-1)$ when $c=\frac{n}{k}-2\epsilon_p$.
It follows that
\begin{align}
J(-1)=& \ \frac{4p}{c}(n-kc)^2 +6k^3c-6nk^2+2k^2c+2kn-\frac{4n^2}{c} \nonumber\\
=& \ \frac{4pk^2}{\frac{n}{k}-2\epsilon_p} 4\epsilon_p^2+\epsilon_p
(-12k^3-4k^2-\frac{8nk}{\frac{n}{k}-2\epsilon_p})  \nonumber\\
=& \ \epsilon_p \Big(-12k^3-4k^2-\frac{8nk}{\frac{n}{k}-2\epsilon_p}+
\frac{16 p k^2 \epsilon_p}{\frac{n}{k}-2\epsilon_p} \Big).
\end{align}
Recall that $\epsilon_p=\frac{2n}{2p+2k-1}$. A tedious calculation will
lead to the fact that for $k$ suitably large ($k \geq 5$), $J(-1)>0$ for $p=p(n,
k)$ large enough. On the other hand, if $k$ is small ($k \leq 2$ for
example), then $J(-1)<0$.

The crucial observation leads to the proof of the lemma is that, for
any $0< c \leq \frac{n}{k}-2\epsilon_p$,  we can find $p_1=p(n, k,
c)$ so that $J^{\prime}(x)>0$ for any $x \in [-1, x_0]$.

First let us estimate $x_0$, the unique zero of $P(x)$ on $(-1,0)$. We
have
\begin{align}
P(x_0)&  = \ \frac{(2p-1)k^2 c}{p} x_0^{2p} + 2nk x_0^{2p-1}+ \frac{c}{p}
k^2 \label{root 1}\\
&  = \ x_0^{2p-1} \Big[(2-\frac{1}{p})k^2 c x + 2nk \Big] + \frac{c}{p}
k^2. \label{root 2}
\end{align}

Note that (\ref{root 1}) implies that $2nk (-x_0)^{2p-1}
>\frac{c}{p}k^2$, while (\ref{root 2}) implies that
\[
(-x_0)^{2p-1} \Big[2nk -(2-\frac{1}{p})k^2 c \Big]  <\frac{c}{p}k^2.
\]

To sum up, we have the following estimates on $x_0$
\begin{equation}
(\frac{ck}{2pn})^{\frac{1}{2p-1}}<-x_0<(\frac{ck}{n})^{\frac{1}{2p-1}}.
\label{x0 estimate}
\end{equation}

Next we compute $J^{\prime}(x)$:
\begin{align}
\frac{1}{p} J^{\prime}(x) \ =& \ -4k^3 c(1-\frac{1}{p})(2+\frac{1}{p})
x^{2p}-4nk^2 (2+\frac{1}{p}) x^{2p-1} \nonumber\\
& \ +4k^2c (2+\frac{1}{p})
x+\Big(8nk-\frac{2kn}{p}-\frac{4k^3c}{p^2}-\frac{2k^3c}{p}\Big)
\label{J prime formula 1}.
\end{align}

In order to do calculation in the $O(\frac{1}{p})$ order, we note
that when $p=p(n, k)$ is large enough, we have
\begin{equation}
\frac{2n}{3p}<\epsilon_p <\frac{n}{p} . \label{epsilon estimate}
\end{equation}
Plugging into $c=\frac{n}{k}-s$ where $2\epsilon_p \leq
s<\frac{n}{k}$, it follows from (\ref{J prime formula 1}) that
\begin{align}
\frac{1}{p} J^{\prime}(x)=& \ -4k^3
(\frac{n}{k}-s)(2-\frac{1}{p}-\frac{1}{p^2})
x^{2p}-4nk^2 (2+\frac{1}{p}) x^{2p-1} \nonumber\\
&\ +4k^2 (\frac{n}{k}-s) (2+\frac{1}{p})
x+\Big(8nk-\frac{2kn}{p}-\frac{4 n k^2}{p^2}+ \frac{4 k^3
s}{p^2}-\frac{2n k^2}{p}+\frac{2k^3 s}{p}\Big)
\\
\geq & \ -8 n k^2 x^{2p-1} (x+1)+8nk(1+x) \nonumber \\
& \ +4k^3 \frac{n}{kp} x^{2p}-4 \frac{nk^2}{p} x^{2p-1}+4k^2
(-2s+\frac{n}{kp})x-\frac{2kn}{p}-\frac{2n k^2}{p}+O(\frac{1}{p^2})
\label{J prime formula 2}
\end{align}

It follows from (\ref{x0 estimate}), (\ref{epsilon estimate}), and
(\ref{J prime formula 2}) that
\begin{align}
\frac{1}{p} J^{\prime}(x) \ \geq \ 4k^2 \frac{5n}{3p}
\Big[\frac{ck}{2pn} \Big]^{\frac{1}{2p-1}}-\frac{2kn}{p}-\frac{2n
k^2}{p}+O(\frac{1}{p^2})>\frac{2kn(2k-1)}{p}>0 \label{J prime
formula 3}
\end{align}
for any $p>p_1(n, k, c)$ large enough and any $-1 \leq x \leq x_0$. Note
that in (\ref{J prime formula 3}), we have used
\begin{align}
\lim_{p \rightarrow +\infty} \Big[\frac{ck}{2pn}
\Big]^{\frac{1}{2p-1}}=1. \label{limit to 1}
\end{align}

This completes the proof of Lemma \ref{T positive}.
\end{proof}

Let use remark that if we consider $M_{2,1}$, the conclusion of
Theorem \ref{main any class} is not necessarily stronger than that
of Proposition \ref{canonical class}. The point is that the degree
of the generating polynomial $\phi(U)$ in Proposition \ref{any
class} might depend on the $c$ where $[-c, c]$ is the domain of
$\phi(U)$. In particular, the degree $p$ goes to infinity as $c$
approaches to $0$, while the generating function in Proposition
\ref{canonical class} is quartic polynomial. However, we am able to
show that the proof of Proposition \ref{any class} can be used to
establish the path-connectedness of all $U(n)$-invariant K\"ahler
metrics of $H>0$ on any Hirzebruch manifold $M_{n,k}$.

\begin{corollary}
The space of all $U(n)$-invariant K\"ahler metrics of $H>0$ on
$M_{n,k}$ is path-connected.
\end{corollary}

\begin{proof}
In view of Corollary \ref{convexity}, it suffices to show that given
any $0<c_1<c_2<\frac{n}{k}$, we can construct a continuous family of
generating functions $\phi_c(U)$ where $U \in [-c,c]$ for any $c_1
\leq c \leq c_2$.

Such a family can be constructed following the proof of Proposition
\ref{any class}. In particular, if we examine (\ref{condition 1}),
(\ref{condition 2}), (\ref{condition 4}), (\ref{derivative
positive}), and (\ref{limit to 1}), we conclude by the continuous
dependence of parameters that for any $c \in [c_1, c_2]$ there
exists an sufficiently large integer $p=p(n,k)$ which is independent
of the choice of $c$, $\delta_1=\delta_1(p, n, k, c)$, and
$\delta_2=\delta_2(p, n, k, c)$ such that $\phi_c(U)$ defined by
(\ref{polynomial formula}) is a continuous path of K\"ahler metrics
with $H>0$.
\end{proof}

\vspace{0.3cm}

\subsection{Complete K\"ahler-Ricci solitons revisited}

In this subsection, we are interested in complete K\"ahler-Ricci
solitons in the following two cases.

(1) Compact shrinking K\"ahler-Ricci soliton on Hirzebuch manifolds
$M_{n,k} (k<n)$, which is the compactification of the total space of
$H^{-k} \rightarrow \mathbb{CP}^{n-1}$.

(2) Complete noncompact shrinking K\"ahler-Ricci soliton on the
total space of $H^{-k} \rightarrow \mathbb{CP}^{n-1}$ when $k<n$.

These were constructed by Cao \cite{Cao}, Koiso \cite{Koiso},
Feldman-Ilmanen-Knopf \cite{F-I-K}.

\begin{question} \label{shrinkers and H>0}

(1) What can we say about compact shrinking K\"ahler-Ricci soliton with
$H>0$?
(2) Are there any complete noncompact shrinking K\"ahler-Ricci
soliton with $H>0$?

\end{question}

Before answering Question \ref{shrinkers and H>0}, let us review the
construction of Cao, Koiso, Feldman-Ilmanen-Knopf. We follow Koiso's
approach in the compact case and it can be extended to the complete
noncompact case. Recall the K\"ahler metric $\tilde{g}$ on the
compactification of the $\mathbb{C}^{\ast}$-bundle obtained by
$H^{-k} \rightarrow \mathbb{CP}^{n-1}$. In the compact case, We make
the following assumption on $M_{n,k}$:

\begin{assumption} \label{canonical compactification}
K\"ahler metric $\tilde{g}$ on $M_{n,k}$ is in $2\pi c_1(M_{n,k})$.
i.e. there exists a function $f$ on $\hat{L}$ such that
\[
Ric(\tilde{g})-\tilde{\omega_g}=\sqrt{-1}\partial \bar{\partial} f.
\]
\end{assumption}

\begin{proposition}[Consequence of Assumption \ref{basic compactification}
and Assumption \ref{canonical compactification}] \label{Koiso
soliton condition}

Under Assumption \ref{basic compactification} and Assumption
\ref{canonical compactification}, if we further assume
$U_{min}=-D_{min}$, then $f$ in Assumption \ref{canonical
compactification} is given by
\begin{equation}
\frac{d \phi}{dU}+\frac{\phi}{Q}\frac{dQ}{dU}+2U-\phi \frac{df}{dU}
=0. \label{f equation}
\end{equation}
Moreover, we have $U_{max}=D_{max}$ and $Ric(g_0)=g_0$ on $M$.

\end{proposition}

If we further assume that $f$ in Assumption \ref{canonical
compactification} is give by $f=-\alpha U$ for some constant $\alpha $,
then it follows that $\nabla f=-\frac{\alpha}{2}H$ is a holomorphic
vector field. Equation (\ref{f equation}) becomes the shrinking
soliton equation.
\begin{equation}
\frac{d \phi}{dU}+\frac{\phi}{Q}\frac{dQ}{dU}+2U-\alpha \phi=0.
\label{shrinking equation}
\end{equation}

Note that $Q=(1+\frac{k}{n}U)^{n-1}$, so Equation (\ref{shrinking
equation}) takes the form
\begin{equation}
\frac{d \phi}{dU}+\frac{k(n-1)}{n+kU}\phi+2U-\alpha \phi=0.
\label{shrinking equation simple}
\end{equation}

Equation (\ref{shrinking equation simple}) can be solved explicitly:
\begin{equation}
\phi(U)=\frac{2\eta(U,\alpha)}{Q(U)}-\frac{2
e^{\alpha(U-U_{min})}}{Q(U)} \eta(U_{min},\alpha),  \label{formal
further}
\end{equation}
where $\eta(U,\alpha)$ is a polynomial of degree $n$ defined by
\begin{equation}
\int xe^{-\alpha x}Q(x) dx=-e^{-\alpha x} \eta(x,\alpha).
\label{integral}
\end{equation}

In the compact case it follows from Proposition \ref{Koiso soliton
condition} that $U_{min}=-1$ and $U_{max}=1$, therefore $\phi>0$ is
a smooth function on $(-1, 1)$ with $\frac{d \phi}{dU}|_{U=-1}=2$
and $\frac{d \phi}{dU}|_{U=1}=-2$. While in the noncompact case we
set $U_{min}=-1$ and $U_{max}=+\infty$, therefore $\phi>0$ is a
smooth function on $(-1, +\infty)$ with $\frac{d
\phi}{dU}|_{U=-1}=2$.

\begin{theorem}[Koiso \cite{Koiso}, Cao \cite{Cao},
Feldman-Ilmanen-Knopf \cite{F-I-K}] For any integer $1 \leq k<n$,
consider $(\mathbb{CP}^{n-1}, g_0)$ where $g_0$ is the Fubini-Study metric
with $Ric (g_0)=g_0$.

(1) We assume that the shrinking soliton metric  on $M_{n,k}$ is of
the form (\ref{def of g}), and satisfies Assumptions \ref{assump 1},
\ref{basic compactification}, and \ref{canonical compactification}.
Then there exists a unique shrinking K\"ahler-Ricci soliton on each
$M_{n,k}$ when $1 \leq k<n$. It is unique in the sense that the
value $\alpha>0$ in the associated holomorphic vector field $\nabla
f=-\frac{\alpha}{2}H$ is determined by the unique solution of
$\phi(1)=0$. Cao proved that the Ricci curvature of the soliton
metric is positive on $M_{n,k}$ if and only if $k=1$.

(2) There exists a unique complete shrinking K\"ahler-Ricci soliton
on the total space of $L^{k} \rightarrow \mathbb{CP}^{n-1}$ whose
value $\alpha>0$ is determined by the unique solution to
$\eta(-1,\alpha)=0$.
\end{theorem}

Indeed, by (\ref{shrinking equation simple}) and Proposition
\ref{curvature tensors}, it is easy to write down the curvature tensors for
shrinking solitons.
\begin{align*}
&A=-\frac{1}{2}[{\alpha}^2-\frac{2 \alpha
k(n-1)}{n+kU}-\frac{k^2n(n-1)}{(n+kU)^2}]
\phi+(\alpha-\frac{k(n-1)}{n+kU})U+1\\
&B=\frac{[k^2 n-k(n+kU) \alpha]\phi+2k(n+kU)U}{2(n+kU)^2},\\
&C=\frac{2(n+kU)-k^2 \phi}{(n+kU)^2}.
\end{align*}

It is easy to see that along the zero section $U=-1$,
$2B+\sqrt{AC}>0$ implies that
\begin{equation}
\alpha<\alpha_0(n,k) \doteq \frac{(n-2k)(k+1)}{n-k} \label{critical
alpha}
\end{equation}
For example, when $n=2, k=1$, the necessary condition for $H>0$ is
$\alpha<\alpha_0(2,1)=0$, however the corresponding $\alpha$ on
$M_{2,1}$ is the unique positive root of the equation
$e^{2\alpha}(-{\alpha}^2+2)-3{\alpha}^2-4 \alpha-2=0$ ($\alpha
\simeq 0.5276195199$). Therefore the Cao-Koiso shrinking soliton on
$M_{2,1}$ does not satisfy $H>0$.

Next we analyze the noncompact case in more details. Consider the
following polynomial which is of degree $n$ in terms of $\alpha$.
\[
{\alpha}^{n+1}\eta(x,\alpha) = -{\alpha}^{n+1} e^{\alpha x} \int
xe^{-\alpha x} (1+\frac{k}{n}x)^{n-1} dx
\]
Making use of the integration formula
\[
\int z^n e^z dz=(\sum_{l=0}^{n} (-1)^{n-l} \frac{n !}{l !} z^l)
e^z+C,
\] Some routine calculation leads to
\begin{align*}
{\alpha}^{n+1}\eta(U,\alpha)=(\frac{k}{n})^{n-1}\Big(\sum_{l=1}^{n}
\frac{n!}{l!} \frac{(n+kU)^{l-1}(n+kU-l)}{k^{l}}
{\alpha}^{l}+n!\Big).
\end{align*}
Therefore the value of ${\alpha}_{\ast}$ which solves the shrinking
soliton equation, namely, the root of the polynomial
$\eta(-1,\alpha)$, is reduced to root of the following polynomial of
degree $n$
\[
\chi(\alpha)=\sum_{l=1}^{n} \frac{n!}{l!}
\frac{(n-k)^{l-1}(n-k-l)}{k^{l}} {\alpha}^{l}+n!.
\]

Similarly as in Feldman-Ilmanen-Knopf (see the equation $f$ in P197
\cite{F-I-K}), since $\chi(\infty)<0$ and $\chi(0)>0$, Descartes'
rule of signs implies there exists a unique positive root for
$\chi(\alpha)$. Call it ${\alpha}_{\ast}$. For this choice
$\alpha={\alpha}_\ast$, $\phi(U)$ has the asymptotical behavior
$\phi \sim \dfrac{1}{\alpha_{\ast}}U$ as $U \rightarrow \infty$, and
it shows that the soliton metric is complete along infinity.

Now we are interested in a more precise estimate of
${\alpha}_{\ast}$. First note that
${\alpha}_{\ast} > k$. To see this, we check
\begin{align}
\chi(\alpha)&= \ \sum_{l=1}^{n} \frac{n!}{l!}
\frac{(n-k)^{l-1}(n-k-l)}{k^{l}} {\alpha}^{l}+n!  \nonumber \\
&= \ \Big(\frac{\alpha (n-k)}{k}\Big)^n+\sum_{l=0}^{n-1} \frac{n!}{l!}
(\frac{\alpha (n-k)}{k})^l (1-\frac{\alpha}{k}). \label{rough
estimate}
\end{align}

We propose the following conjecture on a precise estimate on
$\alpha_{\ast}$.

\begin{conjecture}  \label{root range}
For any $1 \leq k<n$, ${\alpha}_{\ast}$, which is the unique
positive root of of the polynomial $\chi(\alpha)$, satisfies
$\alpha_0(n,k)<{\alpha}_{\ast}<k+1$. If so, then none of
Feldman-Ilmanen-Knopf shrinking solitons have positive holomorphic
sectional curvature.
\end{conjecture}

Argue similarly as in (\ref{rough estimate}), we can show that
Conjecture \ref{root range} is indeed true if $k<n \leq k^2+2k$. It
is very likely that it is true in general, as numerical experiments
suggest.

On the other hand, compact shrinking solitons on $M_{n,k}$ could
have positive holomorphic sectional curvature as the ratio
$\frac{n}{k}$ grows larger.

\begin{proposition} (Some compact shrinking solitons on $M_{n,k}$ have $H>0$).
If we fix $k=1$, then the lowest dimension example of Cao-Koiso
shrinking solitons which have $H>0$ is $n=3$. On $M_{3,1}$ its local
pinching constant is $\frac{1-\alpha}{(2-\alpha)(5-\alpha)} \simeq
0.05587$ where $\alpha \simeq 0.6820161326$. If $k=2$, the first
compact example with $H>0$ is on $M_{7,2}$, where $\alpha \simeq
1.742423694$.
\end{proposition}

Let us explain the calculation in the case of $M_{3,1}$. In this
case the corresponding $\alpha$ is the unique positive solution of
the following equation
\[
16{\alpha}^3+24{\alpha}^2+18{\alpha}+6-e^{2{\alpha}}
(-4{\alpha}^3+6{\alpha}+6)=0.
\]In particular, $\alpha \simeq 0.6820161326 <\alpha_0 (3,1)=1.5$
where $\alpha_{0}(n,k)$ defined in (\ref{critical alpha}). One can
show that that $(2B+\sqrt{AC})(n+kU)^2$ is increasing on $U \in
[-1,1]$ by a direct calculation, therefore we have $H>0$. In
general, for any $M_{n,1}$ with $n \geq 3$, one could expect to
verify that $\alpha<\alpha_0 (n,k)$ and $(2B+\sqrt{AC})(n+kU)^2$ is
increasing on $U \in [-1,1]$. Therefore the Cao-Koiso shrinking
soliton on Hirzebruch manifold $M_{n,1}$ have $H>0$ for any $n \geq
3$.

Let us calculate the local holomorphic pinching constant of
Cao-Koiso soliton on $M_{3,1}$. A similar argument as in Example
\ref{canonical class} shows that
$$
\frac{\min_{||v||=1} H(U, v)}{\max_{||v||=1} H(U, v)}=
 \left\{ \begin{array}{lll}
\frac{AC-4B^2}{A(A+C-4B)},  &  U \in [-1, U_{\ast}]\\
\mbox{} & \mbox{} \\
\frac{C}{A}, &  U \in [U_{\ast}, 1]
\end{array} \right.
$$
Here $U_{\ast}$ is the solution of $2B=C$ on $[-1,1]$, whose
numerical value is about $-0.5.73003$. Therefore the pinching
constant is obtained at $U=-1$, which is $\frac{1-E}{(2-E)(5-E)}
\simeq 0.05587$.

Let us remark that if we drop the assumption of the shrinking
soliton, it is easy to write down examples of complete K\"ahler
metrics of $H>0$ on the total space of $H^{-k} \rightarrow
\mathbb{CP}^1$. For instance, we have

\begin{example} Define a function
$\omega(U)$ on $[-\frac{1}{2}, +\infty)$ as follows:
$$ \omega(U)= \left\{  \begin{array}{lll}
 \frac{1}{2}-2U^2, &  U \in [-\frac{1}{2}, -\frac{1}{4})\\
 & \\
 c(U+2)+\frac{1}{2}\ln(U+2), & U \in [-\frac{1}{4}, \infty)
 \end{array} \right. $$
where $c=\frac{3}{14}-\frac{2}{7} \ln(\frac{7}{4}) \sim 0.05439$ so
that $\omega$ is continuous at $U=-\frac{1}{4}$. Choose a small
number $\delta >0$ so that $\omega$ admits a convex smoothing $\phi$
which equals to $\omega$ except inside $(-\frac{1}{4}-\delta,
-\frac{1}{4}+\delta)$, Note that for $\omega$, both $C>0$ and
$2B+\sqrt{AC}$ have positive lower bounds in $(-\frac{1}{4}-\delta,
-\frac{1}{4})$ and $(-\frac{1}{4}, -\frac{1}{4}+\delta)$. It
guarantees the existence of a convex smoothing $\phi$ which in turns
gives a complete K\"ahler metric with $H>0$ on the total space of $L^{-1} \rightarrow \mathbb{CP}^1$. The metric actually has positive bisectional curvature outside a compact subset.
Moreover, its bisectional curvatures decay quadratically along the
infinity. Asymptotically such a metric has a conical end and there
exist holomorphic functions with polynomial growth on it.
\end{example}

\subsection{$H>0$ is not preserved along the K\"ahler-Ricci flow}

Recently there have been much progress on K\"ahler-Ricci flow with
Calabi-symmetry on Hirzebruch manifolds $M_{n,k}$. See for example
Zhu \cite{Zhu}, Weinkove-Song \cite{SW2011}, Fong \cite{Fong}, and
Guo-Song \cite{GuoSong}. In this subsection, we apply Hitchin's
example and new examples constructed in Theorem \ref{main any class}
to show that in general $H>0$ is not preserved by the K\"ahler-Ricci
flow.

These results imply that if the initial metric $g_0$ is of
$U(n)$-symmetry and in the K\"ahler class
$\frac{b_0}{k}[E_{\infty}]-\frac{a_0}{k}[E_0]$ where $b_0>a_0>0$,
then the flow always develops a singularity in finite time. Let
$T<\infty$ denote the maximal existence time. Their results can be
summarized as:

(1) Suppose that the initial metric $g_0$ satisfies
$\frac{|E_{\infty}|}{|E_0|}=\frac{b_0}{a_0}=\frac{n+k}{n-k}$, where
$|E_0|$ denotes the volume of the divisor $E_0$ with respect to
$g_0$. In this case the K\"ahler class of $g_0$ is proportional to
the anti-canonical class of $M_{n,k}$. Then the K\"ahler-Ricci flow
shrinks the fiber and the base uniformly and collapses to a point.
The rescaled flow converges to the Cao-Koiso soliton on $M_{n,k}$
(\cite{Zhu}).

(2) If $g_0$ satisfies $\frac{b_0}{a_0}<\frac{n+k}{n-k}$ , then the
K\"ahler-Ricci flow shrinks the fiber first, and the flow collapses
to the base $\mathbb{CP}^{n-1}$ (\cite{SW2011}). The rescaled flow
converges to $\mathbb{C}^{n-1} \times \mathbb{CP}^1$ (\cite{Fong}).

(3) If $g_0$ satisfies $\frac{b_0}{a_0}>\frac{n+k}{n-k}$, then the
K\"ahler-Ricci flow shrinks the zero section $E_0$ first and hence
`contracts the exceptional divisor' (\cite{SW2011}). The rescaled
flow converges to the Feldman-Ilmanen-Knopf shrinking soliton on the
total space of $H^{-k} \rightarrow \mathbb{CP}^{n-1}$
(\cite{GuoSong}).

(4) If $n \leq k$, then the K\"ahler-Ricci flow shrinks the fiber
first, and the flow converges to the base $\mathbb{CP}^{n-1}$ in the
Gromov-Hausdorff sense as $t \rightarrow T$ (\cite{SW2011}).

Based on Hitchin's examples reformulated in Example \ref{other
class} and Proposition \ref{any class}, we have the following
examples of the K\"ahler-Ricci flow.

On $M_{2,1}$, let us take the metric in the anti-canonical class
constructed in Example \ref{canonical class} as the initial metric.
Then the normalized flow converges to the Cao-Koiso soliton.
Unfortunately, the positivity of $H$ breaks down. Therefore, in
general, $H>0$ is not preserved along the K\"ahler-Ricci flow. If
instead we start from initial  metric corresponding to
$\phi_{c}(U)=c-\frac{x^2}{c}$ on $[-c,c]$ s in Example \ref{other
class}, where $0<c<\frac{2}{3}$, then the limiting metric of the
unnormalized flow is $(\mathbb{CP}^1, cg_{FS})$. In this case, it is
not clear whether $H>0$ is preserved, or how the holomorphic
pinching constant of $g(t)$ evolves.

On $M_{3,1}$, if we pick initial metric as the Cao-Koiso shrinking
soliton, it is a fixed point of the normalized flow, therefore the
holomorphic pinching constant remains constant.

On $M_{4,1}$, if we pick initial metric by the examples
$\phi_{c}(U)=c-\frac{x^2}{c}$ on $[-c,c]$ for any $0<c<\frac{4}{3}$. Then all
three cases mentioned above could occur. For $c=1$, the
normalized flow evolves $\phi_1(U)$ to the Cao-Koiso soliton. While
the initial metric has the holomorphic pinching constant $2/27
\simeq 0.074$, the limit metric has the holomorphic pinching
constant $\simeq 0.095$. It indeed improves after a long time,
even though we do not know the short time effect of it. If
$1<c<\frac{4}{3}$, then the normalized flow evolves $\phi_{c}(U)$ to the
limiting Feldman-Ilmanen-Knopf shrinking soliton on the total space
of $L^{-1} \rightarrow \mathbb{CP}^{3}$. However the limiting
soliton no longer has $H>0$, and once again we see that $H>0$ is not
preserved under the K\"ahler-Ricci flow.

Therefore a natural question arises, namely,  is there an
effective way to construct a one-parameter family of deformation of
K\"ahler metrics with $H>0$? It would be ideal if the holomorphic
pinching constant could enjoy some monotonicity properties along this
deformation.

\vspace{0.5cm}

\section{K\"ahler metrics of $H>0$ from the submanifold point of view}

In this section, we discuss holomorphic pinching of the canonical
K\"ahler-Einstein metrics on some K\"ahler $C$-spaces. In general,
we would like to discuss the question of constructing $H>0$ metric
from the submanifold point of view.

\begin{proposition}
Consider the flag threefold, or more
generally, let $M$ be the hypersurface in $\mathbb{CP}^n \times
\mathbb{CP}^n$ defined by
\begin{equation}
\sum_{i=1}^{n+1} z_i w_i = 0,  \label{defining flag}
\end{equation}
where $n\geq 2$ and $([z],[w])$ are the homogeneous coordinates. Let
$g$ be the restriction on $M$ of the product of the Fubini-Study
metrics (each of which has $H=2$). Then the holomorphic sectional
curvature of $g$ is between $2$ and $\frac{1}{2}$. So the
holomorphic pinching constant is $\frac{1}{4}$, which is dimension
free.
\end{proposition}

\begin{proof}

Let us work on the case $n=2$ and in the inhomogeneous coordinates $[1,
z_1, z_2]$ and $[w_1, 1, w_2]$. The hypersurface $M^3
\subset \mathbb{CP}^2 \times \mathbb{CP}^2$ is defined by
$w_1+z_1+z_2w_2=0$ and can be parametrized by
\begin{equation}
(t_1, t_2, t_3) \rightarrow [1, t_1, t_2] \times [-t_1 - t_2 t_3, 1,
t_3] .
\end{equation}

It follows that
\[
\frac{\partial }{\partial t_1}= \frac{\partial }{\partial
z_1}-\frac{\partial }{\partial w_1}, \ \frac{\partial }{\partial
t_2}=\frac{\partial }{\partial z_2}-t_3 \frac{\partial }{\partial
w_1}, \ \frac{\partial }{\partial t_3}=-t_2 \frac{\partial
}{\partial w_1}+\frac{\partial }{\partial w_2}.
\]

Therefore under $\{\frac{\partial }{\partial t_1}, \frac{\partial
}{\partial t_2}, \frac{\partial }{\partial t_3} \}$ the induced
metric $\widetilde{g}$ has the form:
\[
  \begin{pmatrix}
    g_{1 \bar{1}}+h_{1 \bar{1}}  &  g_{1 \bar{2}}+\overline{t_3} h_{1 \bar{1}}
    & \overline{t_{2}}h_{1 \bar{1}}+h_{1 \bar{2}} \\
    g_{2 \bar{1}}+\overline{t_3} h_{1 \bar{1}}  &
    g_{2 \bar{2}}+|t_3|^2 h_{1 \bar{1}}
    & t_3 \overline{t_2} h_{1 \bar{1}}-t_3 h_{1 \bar{2}} \\
    t_2 h_{1 \bar{1}}-h_{2 \bar{1}}  &
      t_2 \overline{t_3} h_{1 \bar{1}}-\overline{t_3} h_{2 \bar{1}}
    & |t_2|^2 h_{1 \bar{1}}+h_{2 \bar{2}}
  \end{pmatrix}
\] where $g_{i \bar{j}}=g(\frac{\partial }{\partial z_i},
\frac{\partial }{\partial \overline{z_j}})$ and $h_{i
\bar{j}}=g(\frac{\partial }{\partial w_i}, \frac{\partial }{\partial
\overline{w_j}})$ where $1 \leq i, j \leq 2$ are Fubini-Study
metrics on two factors $\mathbb{CP}^2$ respectively:
\[
g_{i \bar{j}}=\frac{\delta_{ij}}{1+|z|^2}-\frac{z_j
\overline{z_i}}{(1+|z|^2)^2}, \ \ h_{i
\bar{j}}=\frac{\delta_{ij}}{1+|w|^2}-\frac{w_j
\overline{w_i}}{(1+|w|^2)^2}.
\]

Recall that the curvature tensor of $\widetilde{g}$ is given by the
formula
\[
R_{i \bar{j} k \bar{l}}=-\frac{\partial^2 \widetilde{g}_{k
\bar{l}}}{ \partial {z_i} \partial \overline{z}_j}+\widetilde{g}^{p
\bar{q}} \frac{\partial \widetilde{g}_{k \bar{q}}}{\partial z_i}
\frac{\partial \widetilde{g}_{p \bar{l}}}{\partial \overline{z}_j}.
\]
Note that there is a natural $U(3)$-action on $M$ because of the
defining equation (\ref{defining flag}), which acts transitively.
Therefore it suffices to calculate the curvature tensor of the
induced metric $\widetilde{g}$ at the point $(t_1, t_2,
t_3)=(0,0,0)$. Now pick an orthonormal frame $\{e_1, e_2,
e_3\}=\{\frac{1}{\sqrt{2}} \frac{\partial }{\partial t_1},
\frac{\partial }{\partial t_2}, \frac{\partial }{\partial t_3} \}$.
Then a straight forward calculation shows that the only non
vanishing curvature components under $\{e_1, e_2, e_3\}$ are the
following:
\begin{align*}
R_{1 \bar{1} 1 \bar{1}}=1, \ \ R_{2 \bar{2} 2 \bar{2}}=2, \ \ R_{3
\bar{3} 3
\bar{3}}=2,\\
R_{1 \bar{1} 2 \bar{2}}=R_{1 \bar{1} 3 \bar{3}}=\frac{1}{2}, \ \ \
R_{2 \bar{2} 3 \bar{3}}=-\frac{1}{2}.
\end{align*}

From this, we get that $R_{i \bar{j}}=2\delta_{ij}$ for any $1 \leq i\neq j \leq 3$, so
$\widetilde{g}$ is K\"ahler-Einstein. Once we have all the curvature
components, it is direct to see that $\min_{||X||=1}
H(X)=H(\frac{e_2+e_3}{\sqrt{2}})=\frac{1}{2}$ and $\max_{||X||=1}
H(X)=2$.

The same argument works for the hypersurface defined by
$\sum_{i=1}^{n+1} z_i w_i = 0$ in $\mathbb{P}^n \times \mathbb{P}^n$. It gives the same
pinching constant $\frac{1}{4}$ for any $n \geq 2$.
\end{proof}

However, if we try a similar calculation on other types of bidegree
$(p, q)$ hypersurfaces in $\mathbb{CP}^n \times \mathbb{CP}^n$, the
argument breaks down.

As a simple example, consider the bidegree $(2,1)$ hypersurface
defined by $\sum_{i=0}^{2} z_i^2 w_i = 0$ in $\mathbb{CP}^2 \times
\mathbb{CP}^2$ given by homogeneous coordinates $([z],[w])$.
Consider the following parametrization of the hypersurface:
\[
(t_1, t_2, t_3) \rightarrow [1, t_1, t_2,] \times [-t_1^2-t_2^2 t_3,
1, t_3].
\] and
\[
\frac{\partial }{\partial t_1}= \frac{\partial }{\partial z_1}-2t_1
\frac{\partial }{\partial w_1}, \ \frac{\partial }{\partial
t_2}=\frac{\partial }{\partial z_2}-2 t_2 t_3 \frac{\partial
}{\partial w_1}, \ \frac{\partial }{\partial t_3}=-t_2^2
\frac{\partial }{\partial w_1}+\frac{\partial }{\partial w_2}.
\]

The corresponding $\widehat{g}$ induced from the product of
Fubini-Study metric on $\mathbb{CP}^2 \times \mathbb{CP}^2$ is
\[
\begin{pmatrix}
    g_{1 \bar{1}}+4|t_1|^2h_{1 \bar{1}}  &
    g_{1 \bar{2}}+4 t_1 \overline{t_2 t_3} h_{1 \bar{1}}
    & 2t_1\overline{t_{2}}^2 h_{1 \bar{1}}-2t_1 h_{1 \bar{2}} \\
    g_{2 \bar{1}}+4\overline{t_1}\overline{t_2 t_3} h_{1 \bar{1}}  &
    g_{2 \bar{2}}+4 |t_2|^2 |t_3|^2 h_{1 \bar{1}}
    & 2 |t_2|^2 \overline{t_2} t_3 \overline{t_2} h_{1 \bar{1}}-2t_2t_3 h_{1 \bar{2}} \\
    t_2^2 \overline{t_1} h_{1 \bar{1}}-2 \overline{t_1} h_{2 \bar{1}}  &
      2 |t_2|^2 t_2 \overline{t_2} \overline{t_3} h_{1 \bar{1}}-2\overline{t_2 t_3}\overline{t_3} h_{2 \bar{1}}
    & |t_2|^4 h_{1 \bar{1}}+h_{2 \bar{2}}
  \end{pmatrix}
\]

If we only calculate the curvature at $(t_1, t_2, t_3)=(0,0,0)$, we
already encounter some negativity of $H$. First note that
$(\frac{\partial }{\partial t_1}, \frac{\partial }{\partial t_2},
\frac{\partial }{\partial t_3})$ is orthonormal at $(0,0,0)$, then
it follows that
\[
R_{1 \bar{1} 1 \bar{1}}=-\frac{\partial^2 \widehat{g}_{1 \bar{1}}}{
\partial {t_1} \overline{t_1}}=-4 h_{1 \bar{1}}-
\frac{\partial^2 g_{1 \bar{1}}}{ \partial {z_1} \overline{z_1}}=-2.
\]

The same problem occurs for a general bidegree $(2,1)$ hypersurface
in $\mathbb{CP}^n \times \mathbb{CP}^n$. Of course, this just means
that for a bidegree $(2,1)$ hypersurface in $\mathbb{CP}^n \times
\mathbb{CP}^n$, the restriction of the product of the Fubini-Study
metric on the hypersurface does not have $H>0$. But presumably there
could be other metrics on it with $H>0$. This is indeed the case,
and we have the following result:

\begin{proposition} \label{proper kahler cone}
Let $M^n$ be any smooth bidegree $(p,1)$ hypersurface in $\mathbb
{CP}^r \times \mathbb {CP}^s$, where $n=r+s-1$, $p\geq 1$, and $r$,
$s\geq 2$. Then $M^n$ admits a K\"ahler metric with $H>0$. Morever,
when $p>r+1$, the K\"ahler classes of all the K\"ahler metrics with
$H>0$ form a proper subset of the K\"ahler cone of $M^n$.
\end{proposition}

\begin{proof} [Proof of Proposition \ref{proper kahler cone}]
Let $[z]$ and $[w]$ be the homogeneous coordinates of
$\mathbb{CP}^r$ and $\mathbb {CP}^s$, respectively. Let $\pi:
\mathbb{CP}^r \times \mathbb{ CP}^s \rightarrow  \mathbb{CP}^r$ be
the projection map. Suppose that $M^n$ is defined by
\[
\sum_{i=1}^{s+1} f_i(z_1, \cdots, z_{r+1})w_i =0,
\] where each $f_i$ is a homogeneous polynomial of degree $p$.
Consider the sheaf map $h: {\mathcal O}^{\oplus (s+1)} \rightarrow
{\mathcal O}(p)$ on $\mathbb{CP}^r$ defined by
\[ h(e_i)=f_i(z), \ \
\ 1\leq i\leq s+1,
\] where $e_i = (0, \ldots , 0, 1, 0, \ldots , 0)$
has $1$ at the $i$-th position. Clearly, $h$ is surjective, and its
kernel sheaf $E$ is locally free. Since $M^n ={\mathbb P}(E)$ over
$\mathbb {CP}^r$, by the result of \cite{AHZ}, we know that $M^n$
admits K\"ahler metrics with $H>0$.

To see the second part of the statement, let us denote by $H_1$,
$H_2$ the hyperplane section from the two factors restricted on $M$,
then we have $c_1(M)=(r+1-p)H_1+sH_2$. Clearly, $H_1^{r+1}=0$,
$H_2^{s+1}=0$, and since $M^n \sim pH_1+H_2$, we have $H_1^r
H_2^{s-1}=1$ and $H_1^{r-1} H_2^{s}=p$ on $M^n$. For any K\"ahler
class $[\omega ] = aH_1+bH_2$ where $a>0$ and $b>0$, we have
\[
c_1(M)\cdot [\omega ]^{n-1} = a^{r-2}b^{s-2} \Big[ \begin{pmatrix}
    n-1 \\
    r
  \end{pmatrix}  sa^2 +
\begin{pmatrix}
    n-1 \\
    r-1
  \end{pmatrix} (r+1+sp-p)ab + \begin{pmatrix}
    n-1 \\
    r-2
  \end{pmatrix} (r+1-p)b^2 \Big].
\]
So when $p>r+1$ and $b>>a$, we know that the total scalar curvature
of $(M^n,\omega )$ is negative. Thus the K\"ahler classes of metrics
with $H>0$ can not fill in the entire K\"ahler cone.
\end{proof}

For a smooth bidegree $(p,2)$ hypersurface $M^n$ in $\mathbb {CP}^r
\times \mathbb {CP}^s$, where $n=r+s-1$ and $p\geq 2$, one may raise
the question of whether $M^n$ admits K\"ahler metrics with $H>0$?
The answer might be yes in view of Proposition \ref{proper kahler
cone}. Note that if we project to $\mathbb {CP}^r$, then $M$ becomes
a holomorphic fibration over $\mathbb {CP}^r$ whose generic fiber
are smooth quadrics.

\begin{question}
If $M^3$ is a compact K\"ahler manifold with local holomorphic
pinching constant strictly greater $\frac{1}{4}$, then is it
biholomorphic to a compact Hermitian symmetric space, i.e.
$\mathbb{CP}^3$, $\mathbb{CP}^{2} \times \mathbb{CP}^{1}$,
$\mathbb{CP}^{1} \times \mathbb{CP}^{1}\times \mathbb{CP}^{1}$, or
$Q^{3}$ which is the smooth quadric in $\mathbb{CP}^4$?
\end{question}

Let us conclude the discussion here by giving a couple of general
remarks. First, if we we want to construct metrics with $H>0$ from
the submanifold point of view, in particular, as complete
intersections. Then it seems difficult to find examples other than
those already known (such as Hermitian symmetric spaces or K\"ahler
$C$-spaces, or projectivized vector bundles covered in \cite{AHZ}).
For instance, if we consider a cubic hypersurface $M^n \subset
\mathbb {CP}^{n+1}$ and $g$ be the restriction on $M$ of the
Fubini-Study metric. Then it is unclear if $(M^n,g)$ can have $H>0$,
though we expect the answer is no. As another example, if we
consider the restriction of the ambient Fubini-Study metric onto a
complete intersection, where typically we need to restrict to degree
$1$ or $2$. Let us consider $M^n$ as the intersection of two
quadrics, in the case of $n=2$, it is the del Pezzo surface of
degree $4$, or $\mathbb{CP}^2$ blowing up $5$ points. It is unlikely
that the induced metric on $M^n$ could have $H>0$. We plan to
discuss these questions elsewhere.

Secondly, it is a general belief that the existence of a K\"ahler
metric of $H>0$ is a `large open' condition, as illustrated in this
paper on any Hirzebruch manifold. Therefore, it is reasonable to
expect if a projective manifold $M$ admits a K\"ahler metric $g_0$
of $H>0$, then there exsits a small deformation of such a metric
$g_1$ which lies in a Hodge class and still has $H>0$. hence one can
conclude from a theorem of Tian (\cite{Tian}) that $g_1$ can be
approximated by pull backs of Fubini-Study metrics by a sequence of
projective embeddings $\phi_k: M \rightarrow \mathbb{CP}^{N_k}$.
However, it seems difficult to construct examples of K\"ahler
metrics of $H>0$ in this way because of the implicit nature of
$\phi_k$ and $N_k$.

\vskip 0.2cm

\noindent\textbf{Acknowledgments.} We would like to thank Gordon
Heier for helpful suggestions and a few corrections on some
inaccuracies and typos on a preliminary version of this paper, and
Bennett Chow for his interests. Bo Yang is grateful for Xiaodong Cao
for the encouragement, the Math Department at Cornell University for
the excellent working condition, as well as Bin Guo, Zhan Li, and
Shijin Zhang for helpful discussions.

\end{document}